\documentclass[11pt]{article}
\usepackage{amsmath}
\usepackage{amsfonts}
\usepackage{amssymb}
\usepackage{mathrsfs}
\usepackage{amsthm}
\usepackage{palatino, color}
\usepackage[margin=2.5cm, vmargin={1.5cm},includefoot]{geometry}

\definecolor{blue}{rgb}{0,0,1}
\definecolor{red}{rgb}{1,0,0}
\definecolor{green}{rgb}{0,1,0}

\title{Kernels for products of $L$-functions}

\author{Nikolaos Diamantis and Cormac O'Sullivan}
\date{Nov 2, 2011}

\begin{document}

\maketitle

\bibliographystyle{plain}

\def\s#1#2{\langle \,#1 , #2 \,\rangle}

\def\H{{\mathbb H}}
\def\F{{\frak F}}
\def\C{{\mathbb C}}
\def\R{{\mathbb R}}
\def\Z{{\mathbb Z}}
\def\Q{{\mathbb Q}}
\def\N{{\mathbb N}}
\def\B{{\mathbb B}}
\def\G{{\Gamma}}
\def\GH{{\G \backslash \H}}
\def\g{{\gamma}}
\def\L{{\Lambda}}
\def\ee{{\varepsilon}}
\def\K{{\mathcal K}}
\def\Re{\mathrm{Re}}
\def\Im{\mathrm{Im}}
\def\PSL{\mathrm{PSL}}
\def\SL{\mathrm{SL}}
\def\Vol{\operatorname{Vol}}
\def\li{\operatorname{Li}}
\def\sgn{\operatorname{sgn}}
\def\coh{{\mathcal C}}
\def\doh{{\mathcal D}}
\def\lqs{\leqslant}
\def\gqs{\geqslant}

\def\ca{{\frak a}}
\def\cb{{\frak b}}
\def\cc{{\frak c}}
\def\cd{{\frak d}}
\def\ci{{\infty}}

\def\sa{{\sigma_\frak a}}
\def\sb{{\sigma_\frak b}}
\def\sd{{\sigma_\frak d}}
\def\si{{\sigma_\infty}}

\def\ei{{\text{\boldmath $E$}}}
\def\po{{\text{\boldmath $P$}}}
\def\ke{{\mathcal K}}
\def\de{{\mathcal E}}


\newtheorem{theorem}{Theorem}[section]
\newtheorem{lemma}[theorem]{Lemma}
\newtheorem{prop}[theorem]{Proposition}
\newtheorem{cor}[theorem]{Corollary}
\newtheorem{defin}[theorem]{Definition}
\newtheorem{eg}[theorem]{Example}

\renewcommand{\labelenumi}{(\roman{enumi})}

\numberwithin{equation}{section}

\begin{abstract}\noindent
The Rankin-Cohen bracket of two Eisenstein series provides a kernel yielding products of the periods of Hecke eigenforms at critical values. Extending this idea leads to a new type of Eisenstein series built with a double sum. We develop the properties of these series and their non-holomorphic analogs and show their connection to values of $L$-functions outside the critical strip.
\end{abstract}

\section{Introduction}
In 1952, Rankin \cite{R} introduced  the fruitful idea of expressing  the product of two
critical values of the $L$-function of a weight $k$ Hecke eigenform
$f$ for $\G=\SL(2,\Z)$ in terms of
the Petersson scalar product of $f$ and a product
of Eisenstein series:
\begin{equation}\label{rank}
    \s{E_{k_1} E_{k_2}}{f} = (-1)^{k_1/2} 2^{3-k} \frac{k_1 k_2}{B_{k_1} B_{k_2}} L^*(f,1) L^*(f,k_2)
\end{equation}
for $k=k_1+k_2$, the Bernoulli numbers $B_j$ and the completed, entire $L$-function of $f$,
\begin{equation*}
L^*(f, s):=\frac{\G(s)}{(2 \pi)^s} \sum_{m=1}^{\infty}
\frac{a_f(m)} {m^s} = \int_0^\infty f(iy) y^{s-1} \, dy.
\end{equation*}
Zagier  \cite[p. 149]{Za3} extended \eqref{rank} to get
\begin{equation}\label{zag}
    \s{[E_{k_1}, E_{k_2}]_n}{f} = (-1)^{k_1/2}(2\pi i)^n 2^{3-k} \binom{k-2}{n}\frac{k_1 k_2}{B_{k_1} B_{k_2}} L^*(f,n+1) L^*(f,n+k_2)
\end{equation}
where $k=k_1+k_2+2n$ and $[g_1,g_2]_n$ stands for the {\em Rankin-Cohen bracket of index $n$} given by
\begin{equation} \label{rco}
[g_1,g_2]_n := \sum_{r=0}^n (-1)^r \binom{k_1+n-1}{n-r}
\binom{k_2+n-1}{r}g_1^{(r)}g_2^{(n-r)}.
\end{equation}
The {\em periods}  of  $f$ in the critical strip are the numbers
\begin{equation}\label{periods}
L^*(f,1), L^*(f,2), \ldots, L^*(f,k-1).
\end{equation}
 Zagier in \cite[\S 5]{Za3} and Kohnen-Zagier in \cite{KZ} proved important results of
the Eichler-Shimura-Manin theory on the algebraicity of these critical values using
\eqref{zag}. We describe this in more depth in \S\S \ref{valsl}, \ref{per1}.

On the face of it,  the
techniques of \cite{Za3}, employing \eqref{zag}, apply only to
critical values; an extension to non-critical values, $L^*(f,j)$ for integers $j \lqs 0$ or $j \gqs k$, would seem to require Rankin-Cohen
brackets of negative index $n$ or holomorphic Eisenstein series of negative weight, neither of which are defined. Analyzing  the
structure of the Rankin-Cohen bracket of two Eisenstein series in \S \ref{rcbr} reveals a natural construction which we call a {\em double Eisenstein series}\footnote{In the context of multiple zeta functions, the authors in \cite{GKZ} give a different definition of `double Eisenstein series'. See also \cite{Den}, for example, for distinct `double Eisenstein-Kronecker series'. }:
\begin{equation}\label{dbleis1}
    \sum_{\g, \, \delta \in \G_\infty \backslash \G \atop \g\delta^{-1} \neq
\G_\infty}
\left(c_{\g \delta^{-1}}\right)^{l} j(\g,z)^{-k_1} j(\delta,z)^{-k_2}
\end{equation}
where, for $\g \in \G$, we write
 $$
 \g =\left( \smallmatrix a_\g & b_\g \\ c_\g & d_\g \endsmallmatrix \right), \quad j(\g,z):=c_\g z + d_\g.
 $$
 By comparison, the usual holomorphic Eisenstein series  is
\begin{equation}\label{ekzs}
E_k(z):= \sum_{\g \in \G_\infty \backslash \G} j(\g,z)^{-k}.
\end{equation}
The double Eisenstein series \eqref{dbleis1} converges to a weight $k_1+k_2$ cuspform when $l<k_1-2, k_2-2$. For negative integers $l$ it behaves as a Rankin-Cohen
bracket of negative index, see Proposition \ref{th2}. This allows us to further generalize \eqref{rank}, \eqref{zag} and in \S \ref{per} we  characterize the field containing an arbitrary value of an $L$-function in terms of double Eisenstein series and their Fourier coefficients. In the interesting paper \cite{cmz}, Rankin-Cohen
brackets are linked to operations on automorphic pseudodifferential operators and may also be reinterpreted  in this framework allowing for more general indices.

An extension of Zagier's kernel formula \eqref{zag} in the
non-holomorphic direction is given in \S \ref{az}. There we show
 that the holomorphic double Eisenstein series have  non-holomorphic counterparts:
\begin{equation}\label{dgk}
\sum_{\g, \, \delta \in \G_\infty \backslash \G \atop \g\delta^{-1} \neq
\G_\infty} |c_{\g \delta^{-1}}|^{-s-s'}\Im(\g z)^{s} \Im(\delta z)^{s'}.
\end{equation}
These weight $0$ functions possess analytic continuations and functional equations resembling those for the classical non-holomorphic Eisenstein series. As kernels, they produce
 products of $L$-functions for {\em Maass cusp forms}, see Theorem \ref{cpl}. The main
motivation for this construction was its potential use in the rapidly
developing study of periods of Maass  forms \cite{BLZ, LZ, Ma2, Mue}. In developing the properties of \eqref{dgk} we require a certain kernel $\ke(z;s,s')$ as defined in \eqref{nhker}. It is interesting to note that Diaconu and Goldfeld \cite{dg} needed exactly the same series for their results on second moments of $L^*(f,s)$, see \S \ref{pww}.

\vskip 3mm
{\bf Acknowledgements.} We thank Yuri Manin for his stimulating comments on an earlier version of this paper, and the referee who provided the reference \cite{dg}.

\section{Statement of main results}

\subsection{Preliminaries} \label{prel}
Our notation is as in \cite{DO2010}. In all sections but two, $\G$ is the modular group $\SL(2,\Z)$ acting on the upper half plane $\H$. The definitions we give for double Eisenstein series extend easily to more general groups, so in \S \ref{holeis} we prove their basic properties for $\G$ an arbitrary Fuchsian group of the first kind and in \S \ref{slast} we see how some of our main results are valid in this general context.

  Let $S_{k}(\Gamma)$ be the $\C$-vector space of holomorphic, weight $k$ cusp forms for  $\G$ and $M_{k}(\Gamma)$  the  space of modular forms.
These spaces are acted on by the Hecke operators $T_m$, see \eqref{ho}. Let $\mathcal B_k$ be the unique basis of $S_{k}$ consisting of Hecke eigenforms, normalized to have first Fourier coefficient $1$. We assume throughout this paper that $f \in \mathcal B_k$.  Since $\s{T_m f}{f} = \s{f}{T_m f}$ it follows that all the Fourier coefficients of $f$ are real and hence
$
 \overline{L^*(f,s)} = L^*(f,\overline{s})$. Also, recall the functional equation
\begin{equation}\label{lfe}
    L^*(f,k-s)=(-1)^{k/2}L^*(f,s).
\end{equation}

We summarize some standard properties of the non-holomorphic Eisenstein series, see for example \cite[Chapters 3, 6]{Iwsp}. Throughout this paper we use the variables
$
z=x+iy \in \H, \quad s=\sigma + it \in \C.
$
\begin{defin} For $z \in \H$ and $s \in \C$ with $\Re(s) >1$, the weight zero, {\em non-holomorphic Eisenstein series} is
\begin{equation}\label{eis}
    E(z,s):=\sum_{\g \in \G_\infty \backslash \G} \Im(\g z)^s = \frac{y^s}{2}\sum_{c,d \in \Z \atop{(c,d)=1}} |cz+d|^{-2s}.
\end{equation}
\end{defin}
Let $\theta(s):=\pi^{-s}\G(s)\zeta(2s)$.
Then $E(z,s)$ has a Fourier expansion \cite[Theorem 3.4]{Iwsp} which we may write in the form
\begin{equation}\label{fec}
E(z,s) = y^s+\frac{\theta(1-s)}{\theta(s)} y^{1-s}+\sum_{m \neq 0} \phi(m,s) |m|^{-1/2} W_s(mz)
\end{equation}
where $W_s(mz)=2(|m|y)^{1/2}K_{s-1/2}(2 \pi |m|y)e^{2 \pi i m x}$
is the Whittaker function for $z \in \H$ and also
$\theta(s) \phi(m,s) = \sigma_{2s-1}(|m|)|m|^{1/2-s}$. As usual, $\sigma_s(m):=\sum_{d|m} d^s$ is
the divisor function.

For the weight $k \in 2\Z$, non-holomorphic Eisenstein series,
generalizing \eqref{eis}, set
\begin{equation}\label{defe}
 E_k(z,s):=  \sum_{\g \in \G_\ci \backslash \G} \Im (\g z)^s \left(\frac{j(\g, z)}{|j(\g, z)|}\right)^{-k}.
\end{equation}
Then (\ref{defe}) converges to an analytic function of $s\in \C$, and a smooth function of $z\in \H$,  for $\Re(s)>1$. Also $y^{-k/2}E_k(z,s)$ has weight $k$ in $z$. Define the {\em completed
non-holomorphic Eisenstein series} as
\begin{equation}\label{compe}
E^*_k(z,s) := \theta_k(s) E_k(z,s) \quad \text{ for } \quad \theta_k(s) :=  \pi^{-s} \G(s+|k|/2) \zeta(2s).
\end{equation}
With \eqref{fec}, we see that $E(z,s)$ has a meromorphic continuation to all $s \in \C$. The same is true of $E_k(z,s)$, see \cite[\S 2.1]{DO2010} for example.
We have the functional equations
\begin{eqnarray}\label{the}
   \theta(s/2) &=& \theta((1-s)/2),\\
   E^*_k(z,s) &=& E^*_k(z,1-s).\label{eisf}
\end{eqnarray}

\subsection{Holomorphic double Eisenstein series}
Define the subgroup
\begin{equation}\label{gi}
B:=\left\{\left.\left( \smallmatrix 1 & n \\ 0 & 1 \endsmallmatrix \right) \right| n\in\Z \right\} \subset \SL(2,\Z).
\end{equation}
Then $\G_\ci$,   the subgroup of $\G=\SL(2,\Z)$ fixing $\ci$, is $B \cup -B$. For $\g \in \G_\infty \backslash \G$ the quantities $c_\g$, $d_\g$ and $j(\g,z)$ are only defined up to sign (though even powers are well defined). For $\g \in B \backslash \G$ there is no ambiguity in the signs of  $c_\g$, $d_\g$ and $j(\g,z)$.
\begin{defin} Let $z \in \H$ and $w \in \C$. For integers $k_1$, $k_2 \geqslant 3$   we define the {\em double Eisenstein series}
\begin{equation}\label{dbleis1b}
    \ei_{k_1, k_2}(z,w):=
\sum_{\g, \, \delta \in B \backslash \G \atop  c_{\g  \delta^{-1}}  >0}
\left( c_{\g  \delta^{-1}} \right)^{w-1} j(\g,z)^{-k_1} j(\delta,z)^{-k_2}.
\end{equation}
\end{defin}
 This series is well-defined and, as we see in Proposition \ref{cgs}, for $\Re(w) < k_1-1, k_2-1$ converges to a holomorphic function of $z$ that is a weight $k=k_1+k_2$ cusp form. It vanishes identically when $k_1$, $k_2$ have different parity.

Let $k$ be even. To get the most general kernel, with $s\in \C$, set
\begin{equation}\label{dbleis3}
\ei_{s, k-s}(z,w):=
\sum_{ \g, \, \delta \in B \backslash \G \atop  c_{\g  \delta^{-1}}  >0}
\left( c_{\g  \delta^{-1}} \right)^{w-1} \left( \frac{j(\g,z)}{j(\delta,z)} \right)^{-s}j(\delta,z)^{-k}.
\end{equation}
In the usual convention, for $\rho \in \C$ with $\rho \neq 0$ write $\rho=|\rho|e^{i \arg(\rho)}$ for $-\pi < \arg(\rho) \lqs \pi$ and
\begin{equation}\label{rho}
\rho^s=|\rho|^s e^{i \arg(\rho) s} \quad \text{for} \quad s\in \C.
\end{equation}
Note that
$$
c_{\g \delta^{-1}}=\left| \smallmatrix c_\g & d_\g \\ c_\delta & d_\delta \endsmallmatrix \right| >0 \implies \frac{j(\g,z)}{j(\delta,z)} \in \H \quad \text{for} \quad z \in \H
$$
 and so $(j(\g,z)/j(\delta,z))^{-s}$ in \eqref{dbleis3} is well-defined and a holomorphic function of $s\in \C$ and $z \in \H$. Proposition \ref{cgs} shows that $\ei_{s,k-s}(z, w)$ converges absolutely and uniformly on compact sets for which $2<\sigma<k-2$ and $\Re(w)< \sigma-1, k-1-\sigma$.

Define the {\em completed double Eisenstein series} as
\begin{equation}\label{complete}
\ei^*_{s, k-s}(z,w):=\left[\frac{  e^{s i \pi/2} \G(s) \G(k-s) \G(k-w)\zeta(1-w+s)\zeta(1-w+k-s)}{2^{3-w}\pi^{k+1-w}   \G(k-1) }\right] \ei_{s, k-s}(z,w).
\end{equation}
\begin{theorem}\label{maineis} Let $k \gqs 6$ be even. The  series $\ei^*_{s, k-s}(z,w)$ has an analytic continuation to all $s,w \in \C$ and as a function of $z$ is always in $S_k(\G)$. For any  $f$ in $\mathcal B_k$ we have
\begin{equation}\label{lslw}
\s{\ei^*_{s, k-s}(\cdot,w)}{f} =  L^*(f,s) L^*(f,w).
\end{equation}
\end{theorem}

It follows directly from \eqref{lslw} and \eqref{lfe} that $\ei^*_{s, k-s}(z,w)$ satisfies eight functional equations generated by:
\begin{eqnarray}
  \ei^*_{s, k-s}(z,w) &=&  \ei^*_{w, k-w}(z,s), \label{ll2} \\
  \ei^*_{s, k-s}(z,w) &=& (-1)^{k/2} \ei^*_{k-s, s}(z,w). \label{ll1}
\end{eqnarray}

The next result shows how $\ei^*_{s, k-s}$ is a generalization of the Rankin-Cohen bracket $[E_{k_1}, E_{k_2}]_n$.
\begin{prop} \label{th2}
For $n \in \Z_{\gqs 1}$ and even $k_1,k_2 \gqs 4$,
$$
n! [E_{k_1}, E_{k_2}]_n=\frac{2(-1)^{k_1/2}  \pi^{k} \G(k-1)}{(2\pi i)^{n} \zeta(k_1)\zeta(k_2) \G(k_1)\G(k_2) \G(k-n-1)} \ei^*_{k_1+n, k_2+n}(z,n+1).
$$
\end{prop}

Another way to understand these double Eisenstein series is through their connections to non-holomorphic Eisenstein series. Any smooth function, transforming with weight $k$ and with polynomial growth as $y \to \ci$ may be projected into $S_k$ with respect to the Petersson scalar product. See \cite[\S 3.2]{DO2010} and the contained references. Denote this holomorphic projection by $\pi_{hol}$.
\begin{prop}\label{k1k2}
Let  $k=k_1+k_2 \gqs 6$ for  even $k_1,k_2 \gqs 0$. Then for all $s,w \in \C$
$$
\ei^*_{s, k-s}(z,w)=\pi_{hol}\Bigl[ (-1)^{k_2/2} y^{-k/2}E^*_{k_1}(z,  u)E^*_{k_2}(z,  v)/(2\pi^{k/2}) \Bigr]
$$
where
\begin{equation}\label{uv}
u=(s+w-k+1)/2, \quad v=(-s+w+1)/2.
\end{equation}
\end{prop}

\subsection{Values of $L$-functions.} \label{valsl}
For $f \in \mathcal B_k$ let $K_f$ be the field obtained by adjoining to $\Q$ the Fourier coefficients of $f$. We will recall Zagier's proof of the next result in $\S \ref{per1}$.

\begin{theorem} {\bf(Manin's Periods Theorem)}\label{manin}
For each $f \in \mathcal B_k$ there exist $\omega_+(f)$, $\omega_-(f) \in \R$ such that
$$
L^*(f,s)/\omega_+(f), \quad L^*(f,w)/\omega_-(f) \in K_f
$$
for all $s,w$ with $1 \lqs s, w \lqs k-1$ and $s$ even, $w$ odd.
\end{theorem}

Let $m \in \Z$ satisfy $m \lqs 0$ or $m \gqs k$.
Then for these values outside the critical strip we have, according to \cite[\S 3.4]{KonZ} and the
references therein,
$$
L^*(f,m) \in \mathscr P[1/\pi]
$$
where $\mathscr P$ is the ring of periods: complex numbers that may be expressed as an integral of an algebraic function over an algebraic domain.
In contrast to the periods \eqref{periods}, we do not have much more precise information about the algebraic properties of the values $L^*(f,m)$. A special case of a theorem by Koblitz \cite{Ko} shows, for example, that
$$
L^*(f,m) \ \ \not\in \ \ \Z \cdot L^*(f,1)+\Z \cdot L^*(f,2)+ \cdots +\Z \cdot L^*(f,k-1).
$$

Let $K\bigl(\ei^*_{s,k-s}(\cdot,w)\bigr)$ be the field obtained by adjoining to $\Q$ the Fourier coefficients of $\ei^*_{s,k-s}(\cdot,w)$ and let $\omega_+(f)$, $\omega_-(f)$ be as given in Theorem \ref{manin}. Then we have
\begin{theorem} \label{kdkd} For all $f \in \mathcal B_k$ and  $s \in \C$
\begin{align*}
    L^*(f,s)/\omega_+(f) & \in K\bigl(\ei^*_{s,k-s}(\cdot,k-1)\bigr) K_f,\\
    L^*(f,s)/\omega_-(f) & \in K\bigl(\ei^*_{k-2,2}(\cdot,s)\bigr) K_f.
\end{align*}
\end{theorem}
The above theorem gives the link between
 Fourier coefficients of double Eisenstein series and arbitrary values of $L$-functions. We hope that this interesting connection will help shed light on $L^*(f,s)$ for $s$ outside the set $\{1,2,\dots,k-1\}$.  See the further discussion in \S \ref{out} for the case when $s \in \Z$.

In \S \ref{twper} we  also prove  results analogous to Theorem \ref{kdkd} for   the completed $L$-function of $f$ twisted by $e^{2 \pi i m p/q}$ for $p/q\in \Q$:
\begin{equation} \label{Mellin}
L^*(f, s; p/q):=\frac{\G(s)}{(2 \pi)^s} \sum_{m=1}^{\infty}
\frac{a_f(m)e^{2 \pi i m p/q}} {m^s} = \int_0^\infty f(iy+p/q) y^{s-1} \, dy.
\end{equation}

\subsection{Non-holomorphic double Eisenstein series}

\begin{defin} For $z \in \H$, $w$, $s$, $s' \in \C$,   we define the {\em non-holomorphic double Eisenstein series} as
\begin{equation}\label{dbleis2}
    \de(z,w;s,s'):=
\sum_{\g, \, \delta \in \G_\infty \backslash \G \atop \g\delta^{-1} \neq
\G_\infty} \frac {\Im(\g z)^{s} \Im(\delta z)^{s'}}{
|c_{\g \delta^{-1}}|^w }.
\end{equation}
\end{defin}

A simple comparison with \eqref{eis} shows it is absolutely and uniformly convergent for $\Re(s)$, $\Re(s')>1$ and $\Re(w)>0$. (This domain of convergence is improved in Proposition \ref{42}.)
The most symmetric form of \eqref{dbleis2} is when $w=s+s'$. Define
\begin{multline}\label{dbleis20}
\quad \de^*(z;s,s'):= 4\pi^{-s-s'} \G(s)\G(s') \zeta(3s+s')\zeta(s+3s')\de(z,s+s';s,s')\\
+2\theta(s)\theta(s')E(z,s+s'). \quad
\end{multline}
\begin{theorem} \label{cpl}
The completed double Eisenstein series
$\de^*(z;s,s')$ has a meromorphic continuation to all $s,s' \in \C$ and satisfies  the functional equations
 \begin{eqnarray*}
   \de^*(z;s,s') &=& \de^*(z;s',s), \label{tfei1}\\
   \de^*(z;s,s') &=& \de^*(z;1-s,1-s'). \label{tfei2}
 \end{eqnarray*}
For any even  Maass Hecke eigenform $u_j$,
$$
\s{\de^*(z;s,s')}{u_j} = L^*(u_j,s+s'-1/2)L^*(u_j,s'-s+1/2).
$$
\end{theorem}

\section{Further background results and notation}
We need to introduce two more families of modular forms.

\begin{defin} For $z \in \H$, $k \geqslant 4$ in $2 \Z$ and $m \in
\Z_{\geqslant 0}$  the holomorphic {\em Poincar\'e series} is
\begin{equation}\label{poin}
    P_{k}(z;m):=\sum_{\g \in \G_\infty \backslash \G} \frac{e^{2\pi i m
\g z} }{j(\g,z)^{k}} \quad = \quad \frac{1}{2} \sum_{\g \in B \backslash \G} \frac{e^{2\pi i m
\g z} }{j(\g,z)^{k}}.
\end{equation}
\end{defin}
For $m \geqslant 1$ the series $P_{k}(z;m)$ span $S_k(\G)$.
The Eisenstein series $E_k(z)=P_{k}(z;0)$ is not a cusp form but is in the space $M_k(\G)$. The second family of modular forms is based on a series due to Cohen in \cite{Cohen}.

\begin{defin}
The {\em generalized Cohen  kernel} is given by
\begin{equation} \label{cohk}
\coh_k(z,s;p/q) := \frac{1}{2}\sum_{\g \in \G} (\g z+p/q)^{-s} j(\g,z)^{-k}
\end{equation}
for $p/q \in \Q$ and $s \in \C$ with $1<\Re(s)<k-1$.
\end{defin}
In \cite[Section 5]{DO2010} we studied  $\coh_k(z,s;p/q)$  (the factor $1/2$ is included to keep the notation consistent with \cite{DO2010} where $\G=\PSL(2,\Z)$). We showed that, for each $s \in \C$ with $1<\Re(s)<k-1$, $\coh_k(z,s;p/q)$  converges to an element of $S_k(\G)$, with a meromorphic continuation to all $s\in \C$.
From   \cite[Prop. 5.4]{DO2010} we have
\begin{equation}\label{innerc}
\s{\mathcal C_k(\cdot, s;p/q)}{f}= 2^{2-k} \pi e^{-s i \pi/2}
\frac{\G(k-1)}{\G(s) \G(k-s)} L^*(f, k-s; p/q)
\end{equation}
which is a generalization of Cohen's lemma in \cite[\S 1.2]{KZ}.
For simplicity we write $\coh_k(z,s)$ for $\coh_k(z,s;0)$.
The twisted $L$-functions satisfy
\begin{equation}\label{conj2}
 \overline{L^*(f,s; p/q)} = L^*(f,\overline{s}; -p/q)
\end{equation}
and
\begin{equation}\label{ww}
 q^sL^{*}(f, s; p/q) = (-1)^{k/2} q^{k-s} L^{*}(f, k-s; -p'/q)
\end{equation}
for $pp'\equiv 1 \mod q$, as in \cite[App. A.3]{KMV}.

Define
$
\mathcal M_n := \left\{ \left( \smallmatrix a & b \\ c & d \endsmallmatrix \right) \ \Bigl| \ a,b,c,d \in \Z, ad-bc=n\right\}.
$
Thus $\mathcal M_1=\G$. For $k \in \Z$ and $g: \H \to \C$ set
$$(g|_k \g)(z):=\det(\g)^{k/2} g(\g z) j(\g, z)^{-k}$$
for all $\g \in \mathcal M_n$.
The weight $k$ Hecke operator $T_n$ acts on $g \in M_k$ by
\begin{equation}\label{ho}
  (T_n g)(z) := n^{k/2-1} \sum_{\g \in \G \backslash \mathcal M_n} (g|_k \g)(z) = n^{k-1} \sum_{ad=n \atop a,d>0} d^{-k} \sum_{0 \leqslant b <d} g\left( \frac{az+b}{d}\right).
\end{equation}

\section{Basic properties of double Eisenstein series} \label{holeis}
We work more generally in this section with $\G$ a Fuchsian group of the first kind containing at least one cusp. Set
\begin{equation}\label{eg}
\varepsilon_\G := \#\{\G \cap \{-I\}\}.
\end{equation}
Label the finite number of inequivalent cusps $\ca$, $\cb$ etc and let $\G_\ca$ be the subgroup of $\G$ fixing $\ca$. There exists a corresponding scaling matrix $\sa \in \SL(2,\R)$ such that $\sa \ci = \ca$ and
$$
\sa^{-1} \G_\ca \sa = \begin{cases} B \cup -B & \text{ if } -I \in \G \quad (\varepsilon_\G =1)\\
B & \text{ if } -I \not\in \G \quad (\varepsilon_\G =0).
\end{cases}
$$
Also set $\G^*_\ca:=\sa B \sa^{-1}$.

We recall some facts about $E_{k,\ca}(z,s)$, the non-holomorphic Eisenstein series associated to the cusp $\ca$ - see for example \cite[Chap. 3]{Iwsp}, \cite[\S 2.1]{DO2010}. It is defined as
$$
E_{k,\ca}(z,s):=  \sum_{\g \in \G_\ca \backslash \G} \Im (\sa^{-1} \g z)^s \left(\frac{j(\sa^{-1} \g, z)}{|j(\sa^{-1} \g, z)|}\right)^{-k}
$$
and absolutely convergent for $\Re(s)>1$. Put $E^*_{k,\ca}(z,s):=\theta_k(s)E_{k,\ca}(z,s)$ as in \eqref{compe}. Then we have
the expansion
\begin{equation}\label{eoa}
E^*_{0,\ca}(\sb z,s)=\delta_{\ca \cb} \theta(s)y^s + \theta(1-s) Y_{\ca \cb}(s) y^{1-s}+ \sum_{l \neq 0}Y_{\ca \cb}(l,s)W_s(lz)
\end{equation}
and
\begin{equation}\label{yuk}
E^*_{k,\ca}(\sb z,s)=\delta_{\ca \cb} \theta_k(s)y^s + \theta_k(1-s) Y_{\ca \cb}(s) y^{1-s}+ O(e^{-2\pi y})
\end{equation}
as $y \to \ci$ for all $k \in 2\Z$. Also, its functional equation is
\begin{equation}\label{sucu}
E^*_{k,\ca}(z,1-s)=\sum_\cb Y_{\ca \cb}(1-s) E^*_{k,\cb}(z,s).
\end{equation}
We gave the coefficients $Y_{\ca \cb}(s)$ and $Y_{\ca \cb}(l,s)$ explicitly in the case of $\G=\SL(2,\Z)$ following \eqref{fec}, and in general they involve series containing Kloosterman sums, see \cite[(3.21),(3.22)]{Iwsp}.

For the natural generalization of (\ref{dbleis3}) we define the {\em double Eisenstein series associated to the cusp $\ca$} as
\begin{equation}\label{dbleis5}
\ei_{s, k-s, \ca}(z,w):= \sum_{ \g, \, \delta \in \G^*_\ca \backslash \G, \  c_{\sa^{-1}\g  \delta^{-1}\sa}  >0}
\left( c_{\sa^{-1}\g  \delta^{-1}\sa} \right)^{w-1}  \left( \frac{j(\sa^{-1}\g,z)}{j(\sa^{-1}\delta,z)} \right)^{-s} j(\sa^{-1}\delta,z)^{-k}
\end{equation}
so that
\begin{equation}\label{dbleis4}
\ei_{s, k-s, \ca}(\sa z,w)=  j(\sa ,z)^{k}
\sum_{ \g, \, \delta \in B \backslash \G' \atop  c_{\g  \delta^{-1}}  >0}
\left( c_{\g  \delta^{-1}} \right)^{w-1} \left( \frac{j(\g,z)}{j(\delta,z)} \right)^{-s} j(\delta,z)^{-k}
\end{equation}
for $\G'= \sa^{-1}\G \sa$ which is also a Fuchsian group of the first kind.
To establish an initial domain of absolute convergence for \eqref{dbleis4}
 we  consider
 \begin{equation}\label{abseis}
\sum_{\g, \, \delta \in B \backslash \G' \atop  c_{\g  \delta^{-1}}  >0}
\left| \left( c_{\g  \delta^{-1}} \right)^{w-1}  \left( \frac{j(\g,z)}{j(\delta,z)} \right)^{-s} j(\delta,z)^{-k}\right|.
\end{equation}
Recalling \eqref{rho}, we see that
$$
|\rho^s|=|\rho|^\sigma e^{-t \arg(\rho)} \ll_t |\rho|^\sigma \quad \text{for} \quad s=\sigma+it \in \C.
$$
Therefore, with $r=\Re(w)$ and $\Im(\g z)=y |j(\g,z)|^{-2}$ we deduce that \eqref{abseis} is bounded by a constant depending on $s$ times
\begin{equation}\label{abseis2}
 y^{-k/2}\sum_{\g, \, \delta \in \G_\infty \backslash \G' \atop  \g  \delta^{-1} \neq \G_\ci}
\left| c_{\g  \delta^{-1}} \right|^{r-1} \Im(\g z)^{\sigma/2} \Im(\delta z)^{(k-\sigma)/2}.
\end{equation}
\begin{lemma}
There exists a constant $\kappa_\G >0$ so that for all $\g,\delta \in \G$ with $c_{\g  \delta^{-1}}>0$
$$
\kappa_\G \lqs c_{\g  \delta^{-1}} \lqs \Im (\g z)^{-1/2}\Im (\delta z)^{-1/2}.
$$
\end{lemma}
\begin{proof}
The existence of $\kappa_\G$ is described in \cite[\S\S 2.5, 2.6]{Iwsp} and \cite[Lemma 1.25]{S}. Set
$
\ee(\g,z):=j(\g,z)/|j(\g,z)| =e^{i \arg(j(\g, z))}
$.
It is easy to verify that, for all $\g$, $\delta \in  \G$ and $z \in \H$,
\begin{eqnarray}
 c_{\g \delta^{-1}} & =  & c_\g j(\delta,z) - c_\delta j(\g,z) \nonumber\\
 & = & \left(\frac{j(\g,z) - \overline{j(\g,z)}}{2iy} \right)j(\delta,z) - \left(\frac{j(\delta,z) - \overline{j(\delta,z)}}{2iy}\right) j(\g,z) \nonumber\\
   &=& \left(\ee(\delta,z)^{-2} - \ee(\g,z)^{-2}\right) j(\g,z) j(\delta,z)/(2i y). \nonumber
\end{eqnarray}
Therefore
\begin{eqnarray*}
 |c_{\g \delta^{-1}}| & = &  \left|\frac{\ee(\g,z)}{\ee(\delta,z)}-\frac{\ee(\delta,z)}{\ee(\g,z)}\right|
   \Im (\g z)^{-1/2}\Im (\delta z)^{-1/2}/2 \nonumber\\
&=&  \left|\Im\left(\frac{\ee(\g,z)}{\ee(\delta,z)}\right)\right| \Im(\g z)^{-1/2}\Im (\delta z)^{-1/2} \nonumber\\
& \leqslant & \Im (\g z)^{-1/2}\Im (\delta z)^{-1/2}. \label{cgd}
\end{eqnarray*}
\end{proof}
It follows that for $r'=\max(r,1)$ and $\g\delta^{-1} \not\in \G_\infty$
\begin{equation}\label{a}
\left| c_{\g  \delta^{-1}} \right|^{r-1} \ll \Im (\g z)^{(1-r')/2}\Im (\delta z)^{(1-r')/2}
\end{equation}
for an implied constant depending on $\G$ and $r$. Combining \eqref{abseis2}, \eqref{a} shows
\begin{multline}
  \frac{\ei_{s, k-s, \ca}(\sa z,w)}{j(\sa ,z)^{k}}   \ll  y^{-k/2} \sum_{\g, \, \delta \in \G_\infty \backslash \G' \atop \g\delta^{-1} \neq
\G_\infty} \Im(\g z)^{(1-r'+\sigma)/2} \Im(\delta z)^{(1-r'+k-\sigma)/2} \\
 = y^{-k/2}\left[E_\ca\left(\sa z,\frac{1-r'+\sigma}{2}\right) E_\ca\left(\sa z,\frac{1-r'+k-\sigma}{2}\right) - E_\ca\left(\sa z,1-r' + k/2\right) \right] \label{elf}
\end{multline}
on noting that $\Im(\g z)=\Im(\delta z)$ for $\g\delta^{-1} \in
\G_\infty$.
Since $E_{\ca}(z,s)$ is absolutely convergent for
$\sigma=\Re(s)>1$, we have proved that the series $
\ei_{s, k-s, \ca}(\sa z,w)$, defined in \eqref{dbleis4}, is absolutely
 convergent for  $2<\sigma<k-2$ and $\Re(w) < \sigma-1$, $k-1-\sigma$.
This
convergence is uniform for $z$ in compact sets of $\H$ and $s,w$ in compact sets in $\C$ satisfying the above constraints.

We next verify that $\ei_{s, k-s, \ca}(z,w)$ has weight $k$ in the $z$ variable. We have
$$
f(z) \in M_k(\G) \iff f(\sa z) j(\sa,z)^{-k} \in M_k(\sa^{-1}\G \sa)
$$
so with \eqref{dbleis4} we must prove that
$$
g(z):= \sum_{ \g, \, \delta \in B \backslash \G' \atop  c_{\g  \delta^{-1}}  >0}
\left( c_{\g  \delta^{-1}} \right)^{w-1}  \left( \frac{j(\g,z)}{j(\delta,z)} \right)^{-s} j(\delta,z)^{-k}
$$
is in $M_k(\G')$.
For all $\tau \in \G'$
\begin{eqnarray*}
 \frac{g(\tau z)}{j(\tau,z)^{k}} & =  &  \sum_{\g, \, \delta \in B \backslash \G' \atop  c_{\g  \delta^{-1}}  >0}
 \left( c_{\g  \delta^{-1}} \right)^{w-1}  \left( \frac{j(\g, \tau z)}{j(\delta, \tau z)} \right)^{-s} j(\delta, \tau z)^{-k} j(\tau, z)^{-k}\\
 & = &  \sum_{\g, \, \delta \in B \backslash \G' \atop  c_{(\g\tau)  (\delta \tau)^{-1}} >0}
 \left(c_{(\g\tau)  (\delta \tau)^{-1}} \right)^{w-1} \left( \frac{j(\g \tau,z)}{j(\delta \tau,z)} \right)^{-s} j(\delta \tau,z)^{-k} = g(z)
\end{eqnarray*}
as required.

We finally show that $\ei_{s,k-s}$ is a cusp form.
By \eqref{elf}, replacing $z$  by $\sa^{-1}\sb z$ and using \eqref{yuk}, for any cusp $\cb$ we obtain
\begin{multline*}
  \frac{\ei_{s, k-s, \ca}(\sb z,w)}{j(\sb ,z)^{k}} \\
   \ll  y^{-k/2}\left[E_\ca\left(\sb z,\frac{1-r'+\sigma}{2}\right) E_\ca\left(\sb z, \frac{1-r'+k-\sigma}{2}\right) - E_\ca\left(\sb z,1-r' + k/2\right) \right] \\
    \ll  y^{1+\sigma - k} + y^{1-\sigma} + y^{1+r'-k} + y^{r'-k}
\end{multline*}
and approaches $0$  as $y \to \infty$.
 Thus, by a standard argument (see for example \cite[Prop. 5.3]{DO2010}), $\ei_{s, k-s, \ca}(z,w)$ a cusp form.
Assembling these results, we have shown the following:

\begin{prop} \label{cgs}
Let $z \in \H$, $k \in \Z$ and let $s$, $w \in \C$ satisfy $2< \sigma <k-2$ and $\Re(w) < \sigma-1$, $k-1-\sigma$. For $\G$ a Fuchsian group of the first kind with cusp $\ca$, the series $\ei_{s, k-s,\ca}(z,w)$ is
absolutely and uniformly convergent for $s$, $w$ and $z$ in compact sets satisfying the above constraints. For each such $s$, $w$ we have $\ei_{s, k-s, \ca}(z,w)
\in S_{k}(\G)$ as a function of $z$.
\end{prop}

The same techniques prove the next result, for the non-holomorphic double Eisenstein series. Generalizing \eqref{dbleis2}, we set
\begin{equation}\label{dbleis2b}
    \de_\ca(\sa z,w;s,s'):=
\sum_{\g, \, \delta \in \G_\infty \backslash \sa^{-1}\G \sa \atop \g\delta^{-1} \neq
\G_\infty} \frac {\Im(\g z)^{s} \Im(\delta z)^{s'}}{
|c_{\g \delta^{-1}}|^w }.
\end{equation}

\begin{prop} \label{42}
Let  $z \in \H$, $s,s',w \in \C$ with $\sigma=\Re(s)$ and $\sigma'=\Re(s')$. The series $\de_\ca(z,w;s,s')$, defined in \eqref{dbleis2b} is
absolutely and uniformly convergent for  $z$,  $w$, $s$ and $s'$ in compact sets satisfying
$$
\sigma, \sigma'>1 \quad \text{ and } \quad\Re(w) >
2-2\sigma, 2-2\sigma'.
$$
\end{prop}
Unlike $\ei_{s,k-s,\ca}(z,w)$, the series $\de_\ca(z,w;s,s')$ may have polynomial growth at cusps.

\section{Further results on double Eisenstein series}
\subsection{Analytic Continuation} \label{conti}
\begin{proof}[Proof of Theorem \ref{maineis}]
Our next task is to prove the meromorphic continuation of $\ei_{s, k-s}(z,w)$ in $s$ and $w$.  For $s, w$ in
the initial domain of convergence, we begin with
\begin{eqnarray}
  \lefteqn{\zeta(1-w+s)  \zeta(1-w+k-s) \ei_{s, k-s}(z,w)} \quad \quad \quad \quad \nonumber\\
  &=& \sum_{u,v=1}^\infty u^{w-1-s}v^{w-1-k+s}\sum_{a,b,c,d \in \Z \atop{ (a,b)=(c,d)=1 \atop ad-bc > 0}} (ad-bc)^{w-1} \left(\frac{az+b}{c z+d}\right)^{-s} (cz+d)^{-k} \nonumber\\
   &=& \sum_{u,v=1}^\infty \sum_{a,b,c,d \in \Z \atop{ (a,b)=(c,d)=1 \atop ad-bc > 0}} (au \cdot dv-bu \cdot cv)^{w-1} \left(\frac{au \cdot z+bu}{cv \cdot z+dv}\right)^{-s} (cv \cdot z+dv)^{-k} \nonumber\\
   & = & \sum_{a,b,c,d \in \Z \atop ad-bc > 0} (ad-bc)^{w-1} \left(\frac{az+b}{cz+d}\right)^{-s} (cz+d)^{-k} \label{newdef}\\
   &=&   \sum_{n=1}^\infty \frac1{n^{1-w}} \sum_{(\smallmatrix a & b \\ c & d \endsmallmatrix) \in \mathcal M_n}  \left(\frac{az+b}{cz+d}\right)^{-s} (cz+d)^{-k} \nonumber \\
   & = & 2 \sum_{n=1}^\infty \frac{T_n \mathcal C_k(z,s)}{n^{k-w}}, \label{esks}
\end{eqnarray}
recalling \eqref{cohk}. With Proposition \ref{cgs}, we know $\ei_{s, k-s}(z,w) \in S_k(\G)$ so that
\begin{multline*}
  \ei_{s, k-s}(z,w) =\sum_{f \in \mathcal B_k} \frac{\s{\ei_{s, k-s}(\cdot,w)}{f}}{\s{f}{f}} f(z) \\
 \implies \quad
 \zeta(1-w+s)\zeta(1-w+k-s) \ei_{s, k-s}(z,w)  = 2 \sum_{n=1}^\infty \frac1{n^{k-w}} \sum_{f \in \mathcal B_k} \frac{\s{T_n\mathcal C_k(\cdot,s)}{f}}{\s{f}{f}} f(z).
\end{multline*}
Then
$$
  \s{T_n\mathcal C_k(z,s)}{f} = \s{\mathcal C_k(z,s)}{T_nf}
   = a_f(n)\s{\mathcal C_k(z,s)}{f}
$$
and with \eqref{innerc}
we obtain
\begin{multline} \label{key}
\zeta(1-w+s)\zeta(1-w+k-s) \ei_{s, k-s}(z,w) =   2^{3-w}\pi^{k+1-w} e^{-s i \pi/2} \frac{ \G(k-1)}{\G(s) \G(k-s) \G(k-w)}\\
 \times \sum_{f \in \mathcal B_k} L^*(f,k-s) L^*(f,k-w) \frac{f(z)}{\s{f}{f}}.
\end{multline}
Define the completed double Eisenstein series $\ei^*$  with \eqref{complete}. Then \eqref{key} becomes
\begin{equation}\label{key2}
\ei^*_{s, k-s}(z,w) =    \sum_{f \in \mathcal B_k} L^*(f,s) L^*(f,w) \frac{f(z)}{\s{f}{f}}.
\end{equation}
 We also now see from \eqref{key2} that $\ei^*_{s, k-s}(z,w)$ has a analytic continuation to all $s,w$ in $\C$ and satisfies \eqref{lslw} and the two functional equations \eqref{ll2}, \eqref{ll1}. The dihedral group $D_8$ generated by \eqref{ll2}, \eqref{ll1} is described in \cite[\S 4.4]{DO2010}.
\end{proof}

\subsection{Twisted double Eisenstein series}
In this section, we  define  the {\em twisted double Eisenstein series} by
\begin{equation}\label{zpqeis}
\zeta(1-w+s)\zeta(1-w+k-s)\ei_{s, k-s}(z,w;p/q):=  \sum_{a,b,c,d \in \Z \atop ad-bc > 0} (ad-bc)^{w-1} \left(\frac{az+b}{cz+d}+\frac{p}{q}\right)^{-s} (cz+d)^{-k}
\end{equation}
for $ p/q \in \Q$ with $q>0$ and establish its basic required properties. We remark that the above definition of $\ei_{s, k-s}(z,w;p/q)$ comes from generalizing \eqref{newdef}, but it is not clear how it can be extended to general Fuchsian groups.

Writing
$$
(ad-bc)^{w-1} \left(\frac{az+b}{cz+d} +\frac{p}{q}\right)^{-s} = q^{1-w+s}\big((aq+cp)d-(bq+dp)c\big)^{w-1} \left(\frac{(aq+cp)z+(bq+dp)}{cz+d}\right)^{-s}
$$
we see that
$$
\eqref{zpqeis} = q^{1-w+s} \sum_{a',b',c,d \in \Z \atop a'd-b'c > 0} (a'd-b'c)^{w-1} \left(\frac{a'z+b'}{cz+d} \right)^{-s} (cz+d)^{-k}
$$
with $a' \equiv cp \mod q$ and $b' \equiv dp \mod q$. Hence $\ei_{s, k-s}(z,w;p/q)$ is a subseries of $\ei_{s, k-s}(z,w)$ and, in the same domain of initial convergence, is an element of $S_k$.

The analog of \eqref{esks} is
\begin{equation}\label{esks2}
\zeta(1-w+s)\zeta(1-w+k-s)\ei_{s, k-s}(z,w;p/q) = 2 \sum_{n=1}^\infty \frac{T_n \mathcal C_k(z,s;p/q)}{n^{k-w}}.
\end{equation}
Hence, with \eqref{innerc},
\begin{multline} \label{key3}
\zeta(1-w+s)\zeta(1-w+k-s)\ei_{s, k-s}(z,w;p/q) \\=   2^{3-w}\pi^{k+1-w} e^{-s i \pi/2} \frac{ \G(k-1)}{\G(s) \G(k-s) \G(k-w)}
  \sum_{f \in \mathcal B_k} L^*(f,k-s;p/q) L^*(f,k-w) \frac{f(z)}{\s{f}{f}}.
\end{multline}
Define the completed double Eisenstein series $\ei^*_{s, k-s}(z,w;p/q)$  with the same factor as \eqref{complete} and we obtain
\begin{equation}\label{epq}
\s{\ei^*_{s, k-s}(\cdot,w;p/q)}{f} =  L^*(f,k-s;p/q) L^*(f,k-w)
\end{equation}
for any  $f$ in $\mathcal B_k$. Then (\ref{key3}) implies $\ei^*_{s, k-s}(z,w;p/q)$ has an analytic continuation to all $s,w$ in $\C$. It satisfies the two functional equations:
\begin{eqnarray*}
  \ei^*_{s, k-s}(z,k-w; p/q) &=& (-1)^{k/2} \ei^*_{s, k-s}(z,w; p/q),\\
  q^s\ei^*_{k-s, s}(z,w; p/q) &=& (-1)^{k/2} q^{k-s} \ei^*_{s, k-s}(z,w; -p'/q)
\end{eqnarray*}
for $pp'\equiv 1 \mod q$ using \eqref{lfe} and \eqref{ww}, respectively.

\section{Applying the Rankin-Cohen bracket  to Poincar\'e series}\label{rcbr}
The main objective of this section is to show how double Eisenstein series arise naturally when
the Rankin-Cohen bracket is applied to the usual Eisenstein series $E_k$. Proposition \ref{th2}  will be a consequence of this.  In fact, since there is no difficulty in extending these methods, we compute the Rankin-Cohen bracket of two arbitrary Poincar\'e series
$$
\left[P_{k_1}(z;m_1),P_{k_2}(z;m_2)\right]_n
$$
for $m_1,m_2 \gqs 0$. The result may be expressed in terms of the {\em  double Poincar\'e series}, defined below. In this way, the action of the Rankin-Cohen brackets on spaces of modular forms can be completely described. See also Corollary \ref{hjk} at the end of this section.

\begin{defin} Let $z \in \H$,  $k_1$, $k_2 \geqslant 3$ in
$ \Z$ and $m_1$, $m_2 \in \Z_{\geqslant 0}$. For $w \in \C$ with $\Re(w) <
k_1-1, k_2-1$,  we define the {\em  double Poincar\'e series}
\begin{equation}\label{dblpoin}
    \po_{k_1, k_2}(z,w;m_1,m_2):=\sum_{\g, \, \delta \in B
\backslash \G \atop c_{\g\delta^{-1}} >0 } \left(c_{\g\delta^{-1}}\right)^{w-1} \frac{e^{2\pi i (m_1 \g z + m_2 \delta z)}
}{ j(\g,z)^{k_1} j(\delta,z)^{k_2}}.
\end{equation}
\end{defin}
The series \eqref{dblpoin} will vanish identically unless
$k_1$ and $k_2$  have the same parity.
Clearly we have
$
\ei_{k_1, k_2}(z,w)=\po_{k_1, k_2}(z,w;0,0)
$.
Since $|e^{2\pi i (m_1 \g z + m_2 \delta z)}|\lqs 1$, it is a simple matter to verify that the work in \S \ref{holeis} proves that $\po_{k_1, k_2}(z,w;m_1,m_2)$ converges absolutely and uniformly on compacta to a cusp form in $S_{k_1 + k_2}(\G)$.

For $l \in \Z_{\gqs 0}$ it is convenient to set
\begin{equation}\label{qj}
Q_k(z,l;m):=\begin{cases} P_k(z;m) & \text{ if \ } l=0,\\
\frac 12\sum_{\g \in B \backslash \G} \frac{e^{2\pi i m \g z}
 \left(c_{\g }\right)^l}{
j(\g,z)^{k+l}}& \text{ if \ } l \gqs 1.
\end{cases}
\end{equation}

As in the proof of Proposition \ref{cgs}, $Q_k$ is an absolutely convergent series for
$k$ even and at least $4$. The next result may be
verified by induction.
\begin{lemma}\label{lem62}
For every $j \in \Z_{\gqs 0}$, we have the formulas
\begin{eqnarray*}
  \frac{d^j}{dz^j} E_k(z) &=&  (-1)^j
\frac{(k+j-1)!}{(k-1)!}Q_{k}(z,j;0), \\
  \frac{d^j}{dz^j} P_k(z;m) &=& \sum_{l=0}^j (-1)^{l+j} (2 \pi i m)^l
\frac{j!}{l!} \binom{k+j-1}{k+l-1} Q_{k+2l}(z,j-l;m) \qquad (m>0).
\end{eqnarray*}
\end{lemma}
Set
$$
A_{k_1,k_2}(l,u)_n:= \frac{(k_1+n-1)! (k_2+n-1)!}{l! u! (n-l-u)!
(k_1+l-1)! (k_2+u-1)!}.
$$

\begin{prop}\label{rc}
For $m_1,m_2 \in \Z_{\gqs 1}$ 
\begin{multline*}
\left[P_{k_1}(z;m_1),P_{k_2}(z;m_2)\right]_n \\
= \sum_{ l,u \geqslant 0 \atop l+u \leqslant n}
 A_{k_1,k_2}(l,u)_n (-2 \pi i m_1)^l (2 \pi i m_2)^u
 \po_{k_1+n+l-u,k_2+n-l+u}(z,n+1-l-u;m_1,m_2)/2\\
 + P_{k_1+k_2+2n}(z;m_1+m_2) \sum_{ l,u \geqslant 0 \atop l+u = n}
 A_{k_1,k_2}(l,u)_n (-2 \pi i m_1)^l (2 \pi i m_2)^u .
\end{multline*}
\end{prop}
\begin{proof}
With Lemma \ref{lem62}
\begin{multline}\label{first}
\left[P_{k_1}(z;m_1),P_{k_2}(z;m_2)\right]_n \\
= \sum_{l=0}^n \sum_{u=0}^n  (2 \pi i m_1)^l (2 \pi i m_2)^u \frac{(k_1+n-1)! (k_2+n-1)!}{l! u! (k_1+l-1)! (k_2+u-1)!}\\
\times \sum_{r=l}^{n-u} (-1)^{n+l+u+r} \frac{Q_{k_1+2l}(z,r-l;m_1) Q_{k_2+2u}(z,n-r-u;m_2)}{(r-l)!(n-r-u)!}.
\end{multline}
The inner sum over $r$ is
\begin{multline} \label{meq}
\frac{(-1)^l}{4(n-l-u)!}\sum_{\g, \delta \in B \backslash \G}
\frac{e^{2\pi i (m_1 \g z + m_2 \delta z)} }{
j(\g,z)^{k_1+2l}
j(\delta,z)^{k_2+2u}}\\
\times \sum_{r=l}^{n-u}  \binom{n-l-u}{r-l} \left( \frac{c_{\g}}{j(\g,z)}\right)^{r-l} \left( \frac{-c_{\delta}}{j(\delta,z)}\right)^{n-r-u}
\end{multline}
and, employing the binomial theorem, (\ref{meq}) reduces to
\begin{equation}\label{meq2}
\frac{(-1)^l}{4(n-l-u)!}
\sum_{\g, \delta \in B \backslash \G}
\frac{e^{2\pi i (m_1 \g z + m_2 \delta z)} }
{
j(\g,z)^{k_1+n+l-u}
j(\delta,z)^{k_2+n-l+u}}
\big(c_\g j(\delta,z) - c_{\delta}j(\g,z)\big)^{n-l-u}
\end{equation}
for $l+u<n$ and
\begin{equation}\label{meq3}
\frac{(-1)^l}{4(n-l-u)!}\sum_{\g, \delta \in B \backslash \G}
\frac{e^{2\pi i (m_1 \g z + m_2 \delta z)} }{
j(\g,z)^{k_1+n+l-u}
j(\delta,z)^{k_2+n-l+u}}
\end{equation}
for $l+u=n$.
Noting that $c_\g j(\delta,z) - c_{\delta}j(\g,z)=\left| \smallmatrix c_\g & d_\g \\ c_\delta & d_\delta \endsmallmatrix \right| = c_{\g \delta^{-1}}$ means that (\ref{meq2}) becomes
\begin{equation}\label{meq4}
    \frac{(-1)^{l}}{2(n-l-u)!} \po_{k_1+n+l-u,k_2+n-l+u}(z,n+1-l-u;m_1,m_2)
\end{equation}
and (\ref{meq3}) equals
\begin{equation}\label{meq5}
\frac{(-1)^{l}}{(n-l-u)!}\left(\frac{\po_{k_1+n+l-u,k_2+n-l+u}(z,n+1-l-u;m_1,m_2)}{2}
+P_{k_1+k_2+2n}(z;m_1+m_2)
\right).
\end{equation}
Putting (\ref{meq4}) and (\ref{meq5}) into (\ref{first}) finishes the proof.
\end{proof}

In fact, Proposition \ref{rc} is also valid for $m_1$ or $m_2$ equalling
$0$ provided we agree that $(-2 \pi i m_1)^l=1$ in the ambiguous case
where $m_1=l=0$ and similarly that $(2 \pi i m_2)^u=1$ when $m_2=u=0$.
With this notational convention the proof of the last proposition gives
\begin{cor} \label{poi2} For $m >0$ we have
\begin{eqnarray}
\left[E_{k_1}(z),P_{k_2}(z;m)\right]_n
& = & \sum_{u = 0}^{n}
 A_{k_1,k_2}(0,u)_n  (2 \pi i m)^u
 \po_{k_1+n-u,k_2+n+u}(z,n+1-u;0,m)/2 \nonumber\\
& & \quad \quad \quad + P_{k_1+k_2+2n}(z;m) \cdot
 A_{k_1,k_2}(0,n)_n  (2 \pi i m)^n, \nonumber\\
\left[E_{k_1}(z),E_{k_2}(z)\right]_n
& = & A_{k_1,k_2}(0,0)_n
 \ei_{k_1+n,k_2+n}(z,n+1)/2
 + E_{k_1+k_2}(z) \cdot \delta_{n,0}. \label{eisen}
\end{eqnarray}
\end{cor}

Proposition \ref{th2} follows directly from \eqref{eisen}. 
Combining  Proposition \ref{th2}, with  Theorem \ref{maineis} gives a new proof of
Zagier's formula \eqref{zag}. His original proof in \cite[Prop. 6]{Za3} employed Poincar\'e series.

\begin{proof}[Proof of Proposition \ref{k1k2}.]
Let $F_{s,w}(z)=(-1)^{k_2/2} y^{-k/2}E^*_{k_1}(z,  u)E^*_{k_2}(z,  v)/(2\pi^{k/2})$ with $u=(s+w-k+1)/2$, $v=(-s+w+1)/2$ as before in \eqref{uv}. Then $F_{s,w}(z)$ has weight $k$ and polynomial growth as $y \to \ci$. It is proved in \cite[Prop. 2.1]{DO2010} that
\begin{equation}\label{tyu}
\s{F_{s,w}}{f} =  L^*(f,s) L^*(f,w)
\end{equation}
for all $f \in B_k$.
Comparing \eqref{tyu} with \eqref{lslw} shows that $\ei^*_{s, k-s}(\cdot,w) = \pi_{hol}(F_{s,w})$ as required.
\end{proof}

A basic property of Rankin-Cohen brackets  also naturally emerges from Proposition \ref{rc} and Corollary \ref{poi2}:

\begin{cor} \label{hjk}
For $g_1\in M_{k_1}(\G)$ and $g_2 \in M_{k_2}(\G)$ we have
  $[g_1,g_2]_n \in S_{k_1+k_2+2n}(\G)$ for $n>0$.
\end{cor}
\begin{proof}
The space $M_{k_1}(\G)$ is spanned by $E_{k_1}$ and the Poincar\'e series $P_{k_1}(z;m)$ for $m \in \Z_{\gqs 1}$. So we may write $g_1$, and similarly $g_2$, as a linear combination of Eisenstein and Poincar\'e series. Hence $[g_1,g_2]_n$ is a linear combination of the Rankin-Cohen brackets appearing in Proposition \ref{rc} and Corollary \ref{poi2}. By these results $[g_1,g_2]_n$ is a linear combination of double Poincar\'e and double Eisenstein series which  are in $S_{k_1+k_2+2n}(\G)$, as we have already shown.
\end{proof}

It would be interesting to know if $\po_{k_1, k_2}(z,w;m_1,m_2)$ has a meromorphic continuation in $w$. As a corollary of work in the next section we establish the continuation of $\po_{k_1, k_2}(z,w;0,0)$ to all $w \in \C$.

\section{The Hecke action} \label{ax}
The expression  \eqref{esks}, giving $\ei_{s, k-s}$  in terms of $\mathcal C_k$ acted
upon by the Hecke operators,  can be studied further and
yields an interesting relation between $\ei_{s, k-s}(z,w)$ and the generalized Cohen kernel $\coh_k(z,s;p/q)$.

We have
\begin{eqnarray*}
    T_n \mathcal C_k(z,s;p/q) &=&  n^{k-1} \sum_{\rho \in \G \backslash \mathcal M_n} \mathcal C_k(\rho z,s;p/q) \cdot j(\rho, z)^{-k} \\
   &=& \frac 12 n^{k-1} \sum_{\g \in \mathcal M_n} \left(\g z + \frac pq \right)^{-s}  j(\g,  z)^{-k}.
\end{eqnarray*}
To decompose $\mathcal M_n$ into left $\G$-cosets, set $\mathcal H:=\left\{ \left( \smallmatrix a & b \\ 0 & d \endsmallmatrix \right) \ \Bigl| \ a,b,d \in \Z_{\gqs 0}, \ ad=n, \ 0 \lqs b < a\right\}$ so that  $\mathcal M_n = \bigcup_{\rho \in \mathcal H} \rho \G$, a disjoint union.
Hence
\begin{eqnarray}
  T_n \mathcal C_k(z,s;p/q) &=& \frac 12 n^{k-1} \sum_{\rho \in \mathcal H} \sum_{\g \in \G} \left( \rho \g z + \frac pq \right)^{-s} j(\rho, \g z)^{-k} j(\g, z)^{-k} \nonumber\\
&=& \frac 12 n^{k-1} \sum_{a | n} \left(\frac na \right)^{-k} \left(\frac {a^2}n \right)^{-s}\sum_{0 \leqslant b <a}
\sum_{\g \in \G} \left(\g z + \frac{b}{a} + \frac{n}{a^2} \frac pq\right)^{-s} j( \g ,  z)^{-k} \nonumber\\
&=& n^{s-1} \sum_{a | n} a^{k-2s} \sum_{0 \leqslant b <a}
 \coh_k\left(z,s; \frac{b}{a} + \frac{n}{a^2} \frac pq\right). \label{scoh}
\end{eqnarray}
Combining \eqref{scoh} in the case $p/q=0$, with \eqref{esks} we find
\begin{eqnarray*}
\lefteqn{\frac{\zeta(1-w+s)\zeta(1-w+k-s) \ei_{s, k-s}(z,w)}2 = \sum_{n=1}^\infty \frac{T_n \coh_k(z,s)}{n^{k-w}} } \quad \quad \quad \quad \quad \quad \quad \quad  \\
   &=& \sum_{n=1}^\infty  n^{s+w-k-1} \sum_{a|n} a^{k-2s} \sum_{0 \leqslant b <a}
\coh_k\left(z,s;\frac{b}{a}\right) \\
   &=& \sum_{a=1}^\infty a^{k-2s}  \sum_{v=1}^\ci (a v)^{s+w-k-1}  \sum_{0 \leqslant b <a}
\coh_k\left(z,s;\frac{b}{a}\right) \\
   &=& \zeta(k+1-s-w) \sum_{a=1}^\infty a^{w-s-1}    \sum_{0 \leqslant b <a}
\coh_k\left(z,s;\frac{b}{a}\right).
\end{eqnarray*}
Consequently, for $2<\sigma <k-2$ and $\Re(w) <\sigma-1, k-1-\sigma$
\begin{equation}\label{sskv1}
\zeta(1-w+s) \ei_{s, k-s}(z,w) = 2 \sum_{a=1}^\infty a^{w-s-1}    \sum_{b=0}^{a-1}
\coh_k\left(z,s;\frac{b}{a}\right).
\end{equation}

Upon
taking the inner product of both sides   with $f \in \mathcal B_k$, by using
\eqref{lslw}, \eqref{innerc},
and then simplifying we obtain
\begin{equation}\label{abcde}
    \frac{(2\pi)^{k-w}}{\G(k-w)} L^*(f,s) L^*(f,w) = \zeta(k+1-s-w) \sum_{a=1}^\infty a^{w-s-1}   \sum_{b=0}^{a-1}
L^*\left(f,k-s;\frac{b}{a}\right).
\end{equation}
Since the eigenforms $f$ in $\mathcal B_k$ span $S_k$, we may verify \eqref{sskv1} by giving  another proof  of \eqref{abcde}. Note that the right side of \eqref{abcde} equals
\begin{multline*}
  \zeta(k+1-s-w) \frac{\G(k-s)}{(2\pi)^{k-s}} \sum_{a=1}^\infty a^{w-s-1}   \sum_{b=0}^{a-1} \sum_{m=1}^\infty
\frac{a_f(m) e^{2\pi i m b/a}}{m^{k-s}}  \\
  = \zeta(k+1-s-w) \frac{\G(k-s)}{(2\pi)^{k-s}}   \sum_{m=1}^\infty \sum_{a | m}^\infty a^{w-s}
\frac{a_f(m)}{m^{k-s}}   \\
  = \zeta(k+1-s-w) \frac{\G(k-s)}{(2\pi)^{k-s}} \sum_{m=1}^\infty
\frac{a_f(m) \sigma_{w-s}(m)}{m^{k-s}}.
\end{multline*}
The series $$L(f \otimes E(\cdot,v),k-s) := \sum_{m=1}^\infty
\frac{a_f(m) \sigma_{w-s}(m)}{m^{k-s}}$$ is a convolution $L$-series involving the Fourier coefficients of $f(z)$ and $E(z,v)$ for $2v=-s+w+1$ (as in \eqref{uv}) and, recalling \cite[(72)]{Za3} or \cite[(2.11)]{DO2010},
\begin{equation}\label{confrm}
\zeta(k+1-s-w) \frac{\G(k-s)}{(2\pi)^{k-s}} L(f \otimes E(\cdot,v),k-s) = \frac{ (2\pi)^{k-w}}{\G(k-w)}L^*(f,k-s) L^*(f,k-w).
\end{equation}
Applying the functional equation, \eqref{lfe},  confirms that the right side of \eqref{confrm} equals the left side of \eqref{abcde}.

Looking to simplify \eqref{sskv1} leads to the natural question:  what are the relations between the $\coh_k(z,s;p/q)$ for rational $p/q$ in the interval $[0,1)$? For example, it is a simple exercise with \eqref{innerc} and \eqref{ww} to show that
$$
q^{-s}\coh_k(z, s; p/q)=e^{-s i \pi} q^{-k+s}\coh_k(z, k-s; -p'/q)
$$
for $pp' \equiv 1 \mod q$.
With $s=k/2$ at the center of the critical strip  we get an even simpler
relation:
\begin{equation} \label{rel}
\coh_k(z, k/2; p/q)=(-1)^{k/2} \coh_k(z, k/2; -p'/q).
\end{equation}

A more interesting, but speculative, possibility would be to
argue in the reverse direction in order to derive information about
$L$-functions twisted by exponentials with {\em non-rational} exponents.
Specifically, if we established, by other means, relations between the
$\coh_k(z,s; x)$ for $x \not \in \Q$, then \eqref{sskv1} and other results
proven here might lead to relations for $L$-functions twisted by
exponentials with  non-rational exponents. That would be important
because such $L$-functions play a prominent role in Kaszorowski and
Perelli's programme of classifying the Selberg class (see e.g. \cite{KP}).
Relations between these $L$-functions seem to be necessary for the
extension of Kaszorowski and Perelli's classification to degree $2$, to
which $L$-functions of GL$(2)$ cusp forms belong.

\section{Periods of cusp forms} \label{per}
\subsection{Values of $L$-functions inside the critical strip} \label{per1}
We first review Zagier's  proof in \cite[\S 5]{Za3} of Manin's Periods Theorem. This exhibits a general principle of proving algebraicity we will
be using in the next
sections.

For all $s,w \in \C$ it is convenient to define $H_{s,w} \in S_k$ by the conditions
$$
\s{H_{s,w}}{f}=L^*(f,s) L^*(f,w) \quad \text{for all} \quad f \in \mathcal B_k.
$$
We need the following result.
\begin{lemma} \label{shim}
For $g \in S_k$ with Fourier coefficients in the field $K_g$ and $f \in \mathcal B_k$ with coefficients in $K_f$,
$$
\bigl\langle g, f \bigr\rangle / \bigl\langle f, f \bigr\rangle  \in K_g K_f.
$$
\end{lemma}

\begin{proof}
See Shimura's general result \cite[Lemma 4]{Sh}. It is also a simple extension of \cite[Lemma 4.3]{DO2010}.
\end{proof}

Let $K_{critical}$ be the field obtained by adjoining  to $\Q$ all the Fourier coefficients of
$$
\Bigl\{ H_{s,k-1}, H_{k-2,w} \ \Big| \ 1\lqs s,w \lqs k-1, s \text{ even }, w \text{ odd } \Bigr\}.
$$
Thus, with $f \in \mathcal B_k$ and employing Lemma \ref{shim},
\begin{equation}\label{ll45}
L^*(f,k-1) L^*(f,k-2) = \bigl\langle  H_{ k-1, k-2},f \bigr\rangle = c_{f} \bigl\langle f, f \bigr\rangle
\end{equation}
for $c_f \in K_{critical}K_f$ and the left side of \eqref{ll45} is nonzero because the Euler product for $L^*(f,s)$ converges for $\Re(s)>k/2+1/2$.
Set
\begin{equation}\label{omegas}
\omega_+(f) := \frac{c_f \bigl\langle f, f \bigr\rangle}{L^*(f,k-1)}, \quad \omega_-(f) := \frac{\bigl\langle f, f \bigr\rangle}{L^*(f,k-2)}.
\end{equation}
Then $\omega_+(f)\omega_-(f) = \bigl\langle f, f \bigr\rangle$ and we have:
\begin{lemma} \label{zag1}
For each $f \in \mathcal B_k$
$$
L^*(f,s)/\omega_+(f), \quad L^*(f,w)/\omega_-(f) \in K_{critical}K_f
$$
for all $s,w$ with $1 \lqs s, w \lqs k-1$ and $s$ even, $w$ odd.
\end{lemma}
\begin{proof}
For such $s$ and $w$,
\begin{eqnarray*}
  \frac{L^*(f,s)}{\omega_+(f)} &=&  \frac{L^*(f,s)L^*(f,k-1)}{c_f \bigl\langle f, f \bigr\rangle} \ \ = \ \ \frac{\bigl\langle  H_{s, k-1}, f \bigr\rangle}{c_{f} \bigl\langle f, f \bigr\rangle} \ \ = \ \ \frac{c'_{f} \bigl\langle f, f \bigr\rangle}{c_{f} \bigl\langle f, f \bigr\rangle}\in  K_{critical}K_f \\
   \frac{L^*(f,w)}{\omega_-(f)}&=&  \frac{L^*(f,w)L^*(f,k-2)}{c_f \bigl\langle f, f \bigr\rangle} \ \ = \ \ \frac{\bigl\langle  H_{k-2, w}, f \bigr\rangle}{c_{f} \bigl\langle f, f \bigr\rangle} \ \ = \ \ \frac{c''_{f} \bigl\langle f, f \bigr\rangle}{c_{f} \bigl\langle f, f \bigr\rangle}\in  K_{critical}K_f. \qedhere
\end{eqnarray*}
\end{proof}



To deduce Manin's Theorem from Lemma \ref{zag1}, we use Zagier's explicit expression for $H_{s, w}$. For
$n \gqs 0$, even $k_1,k_2 \gqs 4$ and $k=k_1+k_2+2n$, \eqref{zag} implies
\begin{equation}\label{qwet}
(-1)^{k_1/2} 2^{3-k} \frac{k_1 k_2}{B_{k_1} B_{k_2}}\binom{k-2}{n} H_{n+1,n+k_2}=\frac{[E_{k_1}, E_{k_2}]_n}{(2\pi i)^n}.
\end{equation}
 The Fourier coefficients of $E_{k_1}, E_{k_2}$ are rational and hence the right side of \eqref{qwet} has rational coefficients. Then $H_{n+1,n+k_2}$ has Fourier coefficients in $\Q$ (and also for $k_1$, $k_2 = 2$ as described in \cite[p. 214]{KZ}). It follows that  $K_{critical} = \Q$ and Lemma \ref{zag1} becomes Theorem \ref{manin}, Manin's Periods Theorem.

\subsection{Arbitrary $L$-values}\label{out}
With the results of the last section we may now give the proof of Theorem  \ref{kdkd}, restated here:
\begin{theorem} For all $f \in \mathcal B_k$ and $s \in \C$, with $\omega_+(f)$, $\omega_-(f)$ as in Manin's Theorem,
\begin{align*}
    L^*(f,s)/\omega_+(f) & \in K\bigl(\ei^*_{s,k-s}(\cdot,k-1)\bigr) K_f,\\
    L^*(f,s)/\omega_-(f) & \in K\bigl(\ei^*_{k-2,2}(\cdot,s)\bigr) K_f.
\end{align*}
\end{theorem}
\begin{proof}
By Theorem \ref{maineis}, we have $H_{s,w}(z)=\ei^*_{s,k-s}(z,w)$ for all $s$, $w \in \C$. Thus, arguing as in Lemma \ref{zag1} with $\ei^*_{s,k-s}(\cdot,k-1)=H_{s,k-1}$ and $\ei^*_{k-2,2}(\cdot,s)=H_{k-2,s}$ yields the theorem.
\end{proof}

We indicate briefly how the double Eisenstein series Fourier coefficients required to define $K\bigl(\ei^*_{s,k-s}(\cdot,k-1)\bigr)$ and $K\bigl(\ei^*_{k-2,2}(\cdot,s)\bigr)$ in Theorem \ref{kdkd} may be calculated when $s\in \Z$, using a slight extension of the methods in
\cite[\S 3]{DO2010}. We wish to find the $l$-th Fourier coefficient, $a_{s,w}(l)$, of $H_{s,w}(z)=\ei^*_{s,k-s}(z,w)$ for $s$ even and $w$ odd (and we assume $s,w \gqs k/2>1$).
With Proposition \ref{k1k2}, this is $(-1)^{k_2/2}/(2\pi^{k/2})$ times the $l$-th Fourier coefficient of
$$
\pi_{hol}\Bigl[  y^{-k/2}E^*_{k_1}(z,  u)E^*_{k_2}(z,  v) \Bigr]
$$
for
$
 u=(s+w-k+1)/2  $ and  $  v=(-s+w+1)/2  $ both in $\Z$.
 Let
\begin{multline*}
F(z):=y^{-k/2} E^*_{k_1}(z,u) E^*_{k_2}(z,v)  \\
-\frac{\theta_{k_1}(u)\theta_{k_2}(1-v)}{\theta_{k}(s+1-k/2)}y^{-k/2} E^*_k(z,s+1-k/2) -\frac{\theta_{k_1}(u)\theta_{k_2}(v)}{\theta_{k}(w+1-k/2)}y^{-k/2} E^*_k(z,w+1-k/2).
\end{multline*}
Then  $\pi_{hol}\left( y^{-k/2} E^*_{k_1}(z,u) E^*_{k_2}(z,v) \right) = \pi_{hol}\left( F(z) \right)$ because $\pi_{hol}\left( y^{-k/2} E^*_k(z,s)\right) =0$ for every $s$. We have constructed $F$   so that $F(z) \ll y^{-\epsilon}$ as $y \to \ci$ and we may use \cite[Lemma 3.3]{DO2010} to obtain
$$
a_{s,w}(l)=\frac{(-1)^{k_2/2}(4\pi l)^{k-1}}{(2\pi^{k/2}) (k-2)!}\int_0^\infty F_l(y) e^{-2 \pi l y} y^{k-2} \, dy,
$$
on writing $F(z)=\sum_{l \in \Z} e^{2 \pi i l x}y^{-k/2} F_l(y)$. The functions $F_l(y)$ are sums involving the Fourier coefficients of $E^*_{k_1}(z,  u)$ and $E^*_{k_2}(z,  v)$ with $u,v \in \Z$. As shown in \cite[Theorem 3.1]{DO2010} these coefficients are simply expressed in terms of  divisor functions, Bernoulli numbers and a combinatorial part. For $s,w$ in the critical strip, this calculation yields an explicit finite formula for $a_{s,w}(l)$ in \cite[Theorem 1.3]{DO2010} (and another proof that $H_{s,w}$ in \eqref{qwet} has rational Fourier coefficients and that $K_{critical} = \Q$). For $s,w$ outside the critical strip, we obtain infinite series representations for $a_{s,w}(l)$, but again involving nothing more complicated than divisor functions and Bernoulli numbers.
Further details of this computation will appear in \cite{Oeis}.

\subsection{Twisted Periods}\label{twper}
There is an analog of Manin's Periods Theorem for twisted $L$-functions.
Let $p/q \in \Q$ and let $u$ be an integer with $1\lqs u \lqs k-1$. Manin shows in \cite[(13)]{Ma1} (see also \cite[Chapter 5]{La}) that
$i^u\int_0^{p/q} f(iy) y^{u-1} \, dy$ is an integral linear combination of periods
$i^v\int_0^{\infty} f(iy) y^{v-1} \, dy$ for $v=1, \dots, k-1$. With \eqref{Mellin} this
proves
\begin{equation*}
i^u q^{k-2} L^*(f,u;p/q) \ \ \in \ \
\Z \cdot i L^*(f,1) + \Z \cdot i^2 L^*(f,2) + \cdots + \Z \cdot i^{k-1} L^*(f,k-1).
\end{equation*}
Therefore, Theorem \ref{manin} implies the next result.
\begin{prop} \label{twists} For all $f \in \mathcal B_k$,  $p/q
\in \Q$ and integers $u$ with $1\lqs u \lqs k-1$,
$$
L^*(f,u;p/q) \in K_f(i)\omega_+(f)+K_f(i)\omega_-(f).
$$
\end{prop}

Employing \eqref{epq}, a similar proof to that of Theorem \ref{kdkd} in the last section shows the following.
\begin{prop} For all $f \in \mathcal B_k$, $p/q \in \Q$ and $s \in \C$, with $\omega_+(f)$, $\omega_-(f)$ as in Manin's Theorem,
\begin{align*}
    L^*(f,s;p/q)/\omega_+(f) & \in K\bigl(\ei^*_{k-s,s}(\cdot,1;p/q)\bigr) K_f,\\
    L^*(f,s;p/q)/\omega_-(f) & \in K\bigl(\ei^*_{k-s,s}(\cdot,2;p/q)\bigr) K_f.
\end{align*}
\end{prop}

\section{The non-holomorphic case}
\subsection{Background results and notation} \label{pww}
We will need a non-holomorphic analog of the Cohen kernel $\coh_k(z,s)$:
\begin{defin} With $z \in \H$, $s$, $s' \in \C$  define the {\em non-holomorphic kernel} $\ke$ as
\begin{equation}\label{nhker}
\ke(z;s,s'):= \frac 12\sum_{\g \in \G} \frac{\Im(\g z)^{s + s'}}{|\g z|^{2s}}.
\end{equation}
\end{defin}
Following directly from the results  in \cite[\S 5.2]{DO2010}, it is absolutely convergent, uniformly on compacta, for $z \in \H$ and $\Re(s), \Re(s')>1/2$. The kernel $\ke(z;s,s')$ was introduced by Diaconu  and Goldfeld in \cite[(2.1)]{dg}, (though they describe it there as a Poincar\'e series and their kernel is a product of $\Gamma$ factors). Starting with the identity (\cite[Prop. 3.5]{dg})
$$
\s{f\cdot \ke(\cdot;s,s')}{g}=\frac{\G(s+s'+k-1)}{2^{s+s'+k-1}}\int_{-\infty}^{\infty} \frac{L^*(f,\alpha+i \beta) L^*(g,-s+s'+k-\alpha-i \beta)}{\G(s+\alpha+i \beta) \G(-s+s'+k-\alpha-i \beta)}\, d\beta
$$
for $f$, $g$ in $\mathcal B_k$, they provide a new method to establish estimates for the second moment of $L^*(f,s)$ along the critical line $\Re(s)=k/2$. They give similar results for $L^*(u_j,s)$, the $L$-function associated to a Maass form $u_j$ as defined below.

 The spectral decomposition of $\ke(z;s,s')$ and its meromorphic continuation in the $s$, $s'$ variables is shown in \cite[\S 5]{dg}. We do the same; our treatment is slightly different and we include it in $\S 9.2$ for completeness.

For $\G=\SL(2,\Z)$, the discrete spectrum of the Laplace operator $\Delta=-4y^2\partial_z \partial_{\overline z}$ is given by
 $u_0$, the constant eigenfunction, and $u_j$ for $j \in \Z_{\gqs 1}$ an orthogonal system of Maass cuspforms (see e.g. \cite[Chapters 4,7]{Iwsp}) with Fourier expansions
$$
  u_j(z) = \sum_{n \neq 0} |n|^{-1/2} \nu_j(n) W_{s_j}(nz)
$$
where $u_j$ has eigenvalue $s_j(1-s_j)$ and by Weyl's law \cite[(11.5)]{Iwsp}
\begin{equation}\label{weyl}
    \#\{j:|\Im(s_j)| \lqs T\} = T^2/12+O(T \log T).
\end{equation}
We may assume the $u_j$ are Hecke eigenforms
normalized to have $\nu_j(1) = 1$. Necessarily we have $\nu_j(n) \in \R$. Let $\iota$ be the antiholomorphic involution $(\iota u_j)(z):=u_j(-\overline{z})$. We may also assume each $u_j$ is an eigenfunction of this operator, necessarily with eigenvalues $\pm 1$. If $\iota u_j = u_j$ then $\nu_j(n) = \nu_j(-n)$ and $u_j$ is called {\em even}. If $\iota u_j = -u_j$ then $\nu_j(n) = -\nu_j(-n)$ and $u_j$ is  {\em odd}.

The $L$-function associated to the Maass cusp form $u_j$ is $L(u_j,s)=\sum_{n=1}^\infty \nu_j(n)/n^s$, convergent for $\Re(s)>3/2$ since $\nu_j(n) \ll n^{1/2}$ by \cite[(8.8)]{Iwsp}. The completed  $L$-function for an even form $u_j$
is
\begin{equation}\label{maassl}
L^*(u_j,s):=\pi^{-s}\G\left(\frac{s+s_j-1/2}{2}\right)\G\left(\frac{s-s_j+1/2}{2}\right)L(u_j,s)
\end{equation}
and it satisfies
\begin{equation}\label{maassl0}
L^*(u_j,1-s) = L^*(u_j,s)= \overline{L^*(u_j,\overline{s})}.
\end{equation}
  See \cite[p. 107]{Bu} for \eqref{maassl}, \eqref{maassl0} and the analogous  odd case.

To $E(z, s)$ (recall \eqref{fec}) we associate the $L$-function
\begin{equation*}
 L(E(\cdot,s),w):=\sum_{m=1}^\infty \frac{\phi(m,s)}{m^w}
\end{equation*}
The well-known identity $\sum_{m=1}^{\infty} \sigma_x(m)/m^w=\zeta(w)\zeta(w-x)$ implies
\begin{equation}\label{x}
L(E(\cdot,s),w) = \frac{2 \pi^s}{\G(s)}\frac{\zeta(w+s-1/2) \zeta(w-s+1/2)}{\zeta(2s)}.
\end{equation}

\subsection{The non-holomorphic kernel $\ke$}\label{ay}
Throughout this section we use $s=\sigma+it$, $s'=\sigma'+it'$.
Recall $\ke(z;s,s')$ defined in \eqref{nhker} for $\Re(s)$, $\Re(s')>1/2$.
 Our goal is to find the spectral decomposition of $\ke(z;s,s')$ and prove its meromorphic continuation  in $s$ and $s'$. See  \cite[\S 5]{dg} and also \cite[\S 7.4]{Iwsp} for a similar decomposition and continuation of the automorphic Green function.

A routine verification (using \cite[Lemma 9.2]{JO1} for example) yields
\begin{equation}\label{dele}
    \Delta \ke(z;s,s') =(s+s')(1-s-s')\ke(z;s,s') +4ss'\ke(z;s+1,s'+1).
\end{equation}
Put
$$
\xi_{\Z}(z,s):=\sum_{m \in \Z} \frac 1{|z+m|^{2s}}.
$$
Then
\begin{equation}\label{p1}
\ke(z;s,s')= \sum_{\g \in \G_\ci \backslash \G} \Im(\g z)^{s+s'} \xi_{\Z}(\g z,s).
\end{equation}
Use the Poisson summation formula as in \cite[\S 3.4]{Iwsp} or \cite[Th. 3.1.8]{G} to see that
\begin{equation}\label{p2}
\xi_{\Z}(z,s) = \frac{\pi^{1/2} \G(s-1/2)}{\G(s)} y^{1-2s} + \frac{2 \pi^s}{\G(s)} y^{1/2-s} \sum_{m \neq 0} |m|^{s-1/2} K_{s-1/2}(2 \pi |m|y) e^{2\pi imx}
\end{equation}
for $\Re(s)>1/2$. Set
\begin{equation}\label{p3}
\xi^\sharp_{\Z}(z,s) := \sum_{m \neq 0} |m|^{s-1/2} K_{s-1/2}(2 \pi |m|y) e^{2\pi imx}.
\end{equation}
Let $B_\rho:=\{z\in \C: |z| \lqs \rho\}$. Then with \cite[Lemma 6.4]{JO2}
$$
\sqrt{y}K_{s-1/2}(2\pi y) \ll e^{-2\pi y}\left(y^{\rho+3}+y^{-\rho-3} \right)
$$
for all $s \in B_\rho$ and $\rho,y>0$ with the implied constant depending only on $\rho$. Hence
$$
\xi^\sharp_{\Z}(z,s) \ll \sum_{m=1}^\infty e^{-2 \pi my}\left(m^{\rho+\sigma+2}y^{\rho+5/2}+m^{-\rho+\sigma-4}y^{-\rho-7/2} \right).
$$
We also have \cite[Lemma 6.2]{JO2}
$$
\sum_{m=1}^\infty m^\rho e^{-2 m\pi y} \ll e^{-2\pi y}\left(1 + y^{-\rho-1}\right)
$$
for all  $y>0$ with the implied constant depending only on $\rho \gqs 0$. Therefore
\begin{equation}\label{xis}
\xi^\sharp_{\Z}(z,s) \ll e^{-2\pi y}\left(y^{\rho+5/2}+y^{-\rho-9/2}\right).
\end{equation}
Consider the weight $0$ series
\begin{equation}\label{p4}
\ke^\sharp(z;s,s') := \sum_{\g \in \G_\ci \backslash \G} \Im(\g z)^{s'+1/2} \xi^\sharp_{\Z}(\g z,s).
\end{equation}
With \eqref{xis},
we have
\begin{equation}\label{p5}
\ke^\sharp(z;s,s') \ll \sum_{\g \in \G_\ci \backslash \G} \left(\Im(\g z)^{\sigma'+\rho+3}+\Im(\g z)^{\sigma'-\rho-4} \right)e^{-2\pi \Im(\g z)}
\end{equation}
so that $\ke^\sharp(z;s,s')$ is absolutely convergent for $\Re(s') >\rho +5$.

\begin{prop} \label{psiz1}
Let $\rho >0$ and $s,s' \in \C$ satisfy $s \in B_\rho$, $\Re(s) >1/2$ and   $\Re(s') >\rho+5$. Then
\begin{equation}\label{apsi2}
\ke(z;s,s') =  \frac{\pi^{1/2} \G(s-1/2)}{\G(s)} E(z,s'-s+1) + \frac{2 \pi^s}{\G(s)} \ke^\sharp(z;s,s')
\end{equation}
and, for an implied constant depending only on $s,s'$,
\begin{equation}\label{la2}
\ke^\sharp(z;s,s') \ll y^{5+\rho-\sigma'} \quad \text{as} \quad y \to \infty.
\end{equation}
\end{prop}
\begin{proof}
It is clear that \eqref{apsi2} follows from \eqref{p1}, \eqref{p2}, \eqref{p3} and \eqref{p4} when $s$ and $s'$ are in the stated range. With \eqref{p5} and employing \eqref{yuk} we deduce that as $y \to \ci$
\begin{eqnarray*}
  \ke^\sharp(z;s,s')  & \ll & \left(y^{\sigma'+\rho+3}+y^{\sigma'-\rho-4} \right)e^{-2\pi y} + \sum_{\g \in \G_\ci \backslash \G, \g \neq \G_\ci} \left(\Im(\g z)^{\sigma'+\rho+3}+\Im(\g z)^{\sigma'-\rho-4} \right) \\
   & \ll & e^{-\pi y} + y^{1-(\sigma'+\rho+3)} +y^{1-( \sigma'-\rho-4)} \\
   & \ll & y^{5+\rho-\sigma'}.
\end{eqnarray*}
\vspace{-1.3cm}\[\qedhere\]
\end{proof}

Clearly, for $\Re(s') >\rho+5$, \eqref{apsi2} gives the meromorphic continuation of $\ke(z;s,s')$ to all $s \in B_\rho$.
For these $s,s'$ it follows from \eqref{la2} that  $\ke^\sharp$, as a function of $z$,  is bounded.
Also use \eqref{dele} and \eqref{apsi2} to show that
$$
\Delta \ke^\sharp(z;s,s') =(s+s')(1-s-s')\ke^\sharp(z;s,s') +4\pi s'\ke^\sharp(z;s+1,s'+1)
$$
and hence $\Delta \ke^\sharp$ is also bounded.
Therefore, with \cite[Theorems 4.7, 7.3]{Iwsp}, $\ke^\sharp$ has the spectral decomposition
\begin{equation}\label{sd}
\ke^\sharp(z;s,s') = \sum_{j=0}^\infty \frac{\s{\ke^\sharp(\cdot ;s,s')}{u_j}}{\s{u_j}{u_j}} u_j(z)
+\frac{1}{4\pi i} \int_{(1/2)} \s{\ke^\sharp(\cdot ;s,s')}{E(\cdot,r)} E(z,r) \, dr
\end{equation}
where the integral is from $1/2-i\infty$ to $1/2+i\infty$ and the convergence of \eqref{sd}  is pointwise absolute in $z$ and uniform on compacta.

\begin{lemma} \label{l1}
For $s \in B_\rho$ and  $\Re(s') >\rho+5$ we have
$$
  \s{\ke^\sharp(\cdot ;s,s')}{u_j} = \frac{\pi^{1/2-s}}{4  \G(s')} L^*(u_j, s'-s +1/2) \G\left(\frac{s' +s+s_j-1}{2} \right)\G\left(\frac{s' +s-s_j}{2} \right)
$$
when $u_j$ is an even Maass cuspform. If $u_j$ is odd or constant then the inner product is zero.
\end{lemma}
\begin{proof}
Unfolding,
\begin{eqnarray*}
  \s{\ke^\sharp(\cdot ;s,s')}{u_j}  &=& \int_{\GH} \ke^\sharp(z ;s,s') \overline{u_j(z)} \, d\mu(z) \\
   &=& \int_0^\ci \int_0^1 \left( \sum_{m \neq 0} y^{s'+1/2} |m|^{s-1/2} K_{s-1/2}(2 \pi |m|y) e^{2\pi imx} \right)
   \overline{u_j(z)} \, \frac{dx dy}{y^2} \\
   &=& 2\sum_{m \neq 0} \nu_j(m) |m|^{s-1/2} \int_0^\ci y^{s'} K_{s-1/2}(2 \pi |m|y) K_{\overline{s_j}-1/2}(2 \pi |m|y) \frac{dy}{y}.
\end{eqnarray*}
Evaluating the integral \cite[p. 205]{Iwsp} yields
$$
  \s{\ke^\sharp(\cdot ;s,s')}{u_j} = \frac{L(u_j, s'-s +1/2)}{4 \pi^{s'} \G(s')} \ \prod\G\left(\frac{s' \pm (s-1/2) \pm (\overline{s_j} -1/2)}{2} \right).
$$
Using \eqref{maassl} and that $\overline{s_j} = 1-s_j$ finishes the proof.
\end{proof}

In the same way, when $\Re(r)=1/2$,
$$
\s{\ke^\sharp(\cdot ;s,s')}{E(\cdot,r)} = \frac{L(\overline{E(\cdot, r)}, s'-s +1/2)}{4 \pi^{s'} \G(s')} \ \prod \G\left(\frac{s' \pm (s-1/2) \pm (\overline{r} -1/2)}{2} \right).
$$
Further,  $\overline{E(z, r)}=E(z,\overline{r})=E(z,1-r)$ and with \eqref{x} we have shown the following.
\begin{lemma} \label{l2}
For $s \in B_\rho$ and  $\Re(s') >\rho+5$
\begin{multline*}
 \s{\ke^\sharp(\cdot ;s,s')}{E(\cdot,r)}
 = \frac{\pi^{1/2-s}}{2\G(s') \theta(1-r)} \\
 \times \G\left(\frac{s'+s-r}{2}\right)
  \G\left(\frac{s'+s-1+r}{2}\right)  \theta\left(\frac{s'-s+r}{2}\right) \theta\left(\frac{s'-s+1-r}{2}\right).
\end{multline*}
\end{lemma}
Recall that $\theta(s):=\pi^{-s}\G(s)\zeta(2s)$ as in \S \ref{prel}.
Let
\begin{eqnarray*}
   \ke_1(z;s,s') &:=&  \frac{\pi^{1/2} \G(s-1/2)}{\G(s)} E(z,s'-s+1)  \\
     \ke_2(z;s,s') &:=& \frac{\pi^{1/2}}{2 \G(s) \G(s')}\sum_{j=1 \atop u_j \text{ even}}^\infty  L^*(u_j, s'-s +1/2) \G\left(\frac{s' +s+s_j-1}{2} \right)\G\left(\frac{s' +s-s_j}{2} \right) \frac{u_j(z)}{\s{u_j}{u_j}}  \\
    \ke_3(z;s,s') &:=&  \frac{\pi^{1/2}}{\G(s) \G(s')}\frac{1}{4\pi i} \int_{(1/2)}  \G\left(\frac{s'+s-r}{2}\right) \G\left(\frac{s'+s-1+r}{2}\right)  \\
   && \hskip 40mm  \times \theta\left(\frac{s'-s+r}{2}\right) \theta\left(\frac{s'-s+1-r}{2}\right) \frac{E(z,r)}{\theta(1-r)} \, dr.
\end{eqnarray*}
Assembling Proposition \ref{psiz1}, \eqref{sd} and Lemmas \ref{l1}, \ref{l2} we have proven the decomposition
\begin{equation}\label{sdco}
\ke(z;s,s')=\ke_1(z;s,s')+\ke_2(z;s,s')+\ke_3(z;s,s')
\end{equation}
for $s \in B_\rho$ and  $\Re(s') >\rho+5$. This agrees exactly with \cite[(5.8)]{dg}.

Clearly $\ke_1(z;s,s')$ is a meromorphic function of $s$ and $s'$ in all of $\C$.
The same is true for $\ke_2(z;s,s')$ since the factors $L(u_j, s'-s +1/2) \frac{u_j(z)}{\s{u_j}{u_j}}$ have at most polynomial growth as $\Im(s_j) \to \ci$ while the  $\G$ factors have exponential decay by Stirling's formula. See \eqref{weyl} and \cite[\S \S 7,8]{Iwsp} for the necessary bounds.
The next result was first established in \cite[\S 5]{dg}.

\begin{theorem} \label{meromcont}
The non-holomorphic kernel $\ke(z;s,s')$ has a meromorphic continuation to all $s,s' \in \C$.
\end{theorem}
\begin{proof}
As we have discussed, $\ke_1(z;s,s')$ and $\ke_2(z;s,s')$ are meromorphic functions of $s,s' \in \C$.  The poles of $\G(w)$ are at $w=0,-1,-2,\dots$ and $\theta(w)$ has poles exactly at $w=0,1/2$ (with residues $-1/2$, $1/2$ respectively).  Therefore, the integral in $\ke_3(z;s,s')$ is certainly an analytic function of $s,s'$ for $\sigma'>\sigma+1/2$ and  $\sigma >1/2$ since the $\G$ and $\theta$ factors have exponential decay as $|r| \to \ci$.
Next consider $s$ fixed (with $\sigma >1/2$) and $s'$ varying.  Consider a point $r_0$ with $\Re(r_0)=1/2$. Let $B(r_0)$ be a small disc centered at $r_0$ and $B(1-r_0)$ an identical disc at $1-r_0$. By deforming the path of integration to a new path $C$ to the left of $B(r_0)$ and to the right of $B(1-r_0)$, we may, by Cauchy's theorem, analytically continue $\ke_3(z;s,s')$ to $s'$ with $s'-s \in B(r_0)$. Let $C_1$ be a clockwise contour around the left side of $B(r_0)$ and $C_2$ be a counter-clockwise contour around the right side of $B(1-r_0)$ so that $C=(1/2)+C_1+C_2$. For $s'-s$ inside $C_1$ (and $1-(s'-s)$ inside $C_2$) we have
$$
\pi^{-1/2}\G(s) \G(s') \cdot \ke_3(z;s,s') = \frac{1}{4\pi i} \int_{C} \ast = \frac{1}{4\pi i} \int_{(1/2)} * + \frac{1}{4\pi i} \int_{C_1} * + \frac{1}{4\pi i} \int_{C_2} *
$$
with $*$ denoting the integrand in the definition of $\ke_3$.
Then
\begin{eqnarray*}
  \frac{1}{4\pi i} \int_{C_1} &=& \frac{-2\pi i}{4\pi i} \left(\operatornamewithlimits{Res}_{r=s'-s} \theta\left(\frac{s'-s+1-r}{2}\right)\right) \G(s)\G(s'-1/2) \frac{\theta(s'-s)}{\theta(1-s'+s)} E(z,s'-s) \\
   &=& \frac 12 \G(s)\G(s'-1/2) \frac{\theta(s'-s)}{\theta(1-s'+s)} E(z,s'-s) \\
   &=& \frac 12 \G(s)\G(s'-1/2)  E(z,s-s'+1).
\end{eqnarray*}
We get the same result for $\frac{1}{4\pi i} \int_{C_2}$ and it follows that for all $s'$ with $\sigma-1/2 < \Re(s') <\sigma+1/2$, the continuation of
$\ke_3(z;s,s')$ is given by
\begin{equation}\label{contg}
\pi^{-1/2}\G(s) \G(s') \cdot \ke_3(z;s,s') = \G(s)\G(s'-1/2)  E(z,s-s'+1)
+ \frac{1}{4\pi i} \int_{(1/2)}  *.
\end{equation}
Similarly, as $s'$ crosses the line with real part $\sigma-1/2$, the term $-\G(s-1/2)\G(s')  E(z,s'-s+1)$ must be added to the right side of \eqref{contg}. Thus, for all $s'$ with $1/2<\Re(s')<\sigma-1/2$, the continuation of $\ke(z;s,s')$ is
\begin{equation}\label{contg2}
\ke(z;s,s')  = \frac{\pi^{1/2} \G(s'-1/2)}{\G(s')} E(z,s-s'+1) + \ke_2(z;s,s') + \ke_3(z;s,s').
\end{equation}
Clearly, with \eqref{contg}, \eqref{contg2} we have demonstrated the meromorphic continuation of $\ke(z;s,s')$ to all $s,s' \in \C$ with $\Re(s),\Re(s')>1/2$. The continuation to all $s,s' \in \C$ follows in the same way with further terms in the expression for $\ke(z;s,s')$ appearing from the residues of the poles of $\G\left(\frac{s'+s-r}{2}\right) \G\left(\frac{s'+s-1+r}{2}\right)$ as $\Re(s'+s) \to -\infty$.
\end{proof}

\begin{prop}
We have the functional equation
\begin{equation}\label{funpsi}
\ke(z;s,s')=\ke(z;s',s).
\end{equation}
\end{prop}
\begin{proof}
We may  verify \eqref{funpsi} by comparing \eqref{sdco}  with \eqref{contg2} and using that $\ke_2(z;s,s') = \ke_2(z;s',s)$ by \eqref{maassl0}, and $\ke_3(z;s,s')=\ke_3(z;s',s)$ by \eqref{the}.
There is a second, easier proof: with $S=\left (
 \smallmatrix 0 & -1 \\ 1 & 0 \endsmallmatrix \right )$, replace  $\g$ in \eqref{nhker} by $S\g$.
\end{proof}

\begin{prop} \label{poi} For  all $s,s' \in \C$ and any even  Maass Hecke eigenform $u_j$,
$$
\s{\ke(\cdot ;s,s')}{u_j} = \frac{\pi^{1/2}}{2 \G(s)  \G(s')} \G\left(\frac{s' +s+s_j-1}{2} \right)\G\left(\frac{s' +s-s_j}{2} \right) \cdot L^*(u_j, s'-s +1/2).
$$
\end{prop}
\begin{proof}
Since each $u_j$ is orthogonal to Eisenstein series we have by \eqref{sdco} (for $s \in B_\rho$ and  $\Re(s') >\rho+5$) that
$$
\s{\ke(\cdot ;s,s')}{u_j} = \s{\ke_2(\cdot ;s,s')}{u_j}.
$$
The result follows, extending to all $s,s' \in \C$ by analytic continuation.
\end{proof}

\subsection{Non-holomorphic double Eisenstein series}\label{az}

A similar argument to the proof of \eqref{esks} shows that, for $\Re(s)$, $\Re(s')>1$ and $\Re(w)\gqs 0$,
\begin{equation}\label{zze}
\zeta(w+2s) \zeta(w+2s') \de(z,w;s,s')= \frac{1}2 \sum_{n=1}^\infty \frac{T_n \ke(z;s,s')}{n^{w-1/2}}
\end{equation}
where, in this context \cite[(3.12.3)]{G}, the appropriately normalized Hecke operator acts as
$$
T_n \ke(z) = \frac{1}{n^{1/2}} \sum_{\g \in \G \backslash \mathcal M_n} \ke( \g z).
$$
For each Maass form we have $T_n u_j = \nu_j(n) u_j$ and for the Eisenstein series \cite[Prop.
3.14.2]{G} implies $
T_n E(z,s) = n^{s-1/2} \sigma_{1-2s}(n) E(z,s)
$.
Therefore, as in \eqref{x},
$$
  \sum_{n=1}^\infty \frac{T_n E(z,s)}{n^{w-1/2}} = E(z,s)\sum_{n=1}^\infty \frac{\sigma_{1-2s}(n)}{n^{w-s}} =
    E(z,s)\zeta(w-s)\zeta(w+s-1).
$$
Now choose any $\rho >0$. For $s \in B_\rho$, $\Re(s)>1$, $\Re(s')>\rho+5$ and $\Re(w)\gqs 0$ we may
apply $T_n$ to both sides of \eqref{sdco} and obtain
\begin{multline}
  \lefteqn{\zeta(w+2s) \zeta(w+2s') \de(z,w;s,s') = \frac{\pi^{1/2} \G(s-1/2)}{2\G(s)} \zeta(s'-s+w)\zeta(s-s'+w-1)E(z,s'-s+1)} \\
   + \frac{\pi^{1/2}}{4 \G(s) \G(s')}\sum_{j=1 \atop u_j \text{ even}}^\infty  L^*(u_j, s'-s +1/2) \G\left(\frac{s' +s+s_j-1}{2} \right)\G\left(\frac{s' +s-s_j}{2} \right) L(u_j,w-1/2) \frac{u_j(z)}{\s{u_j}{u_j}}
  \\
    + \frac{\pi^{1/2}}{2\G(s)\G(s')} \frac{1}{4\pi i} \int_{(1/2)} \theta\left(\frac{s'-s+r}{2}\right) \theta\left(\frac{s'-s+1-r}{2}\right) \G\left(\frac{s'+s-r}{2}\right) \G\left(\frac{s'+s-1+r}{2}\right)\\
    \times
     \zeta(w-r) \zeta(w-1+r) \frac{E(z,r)}{\theta(1-r)} \, dr. \label{ml}
\end{multline}
Put
$$
\Omega(s,s';r):=\left. \theta\left(\frac{s'+s-r}{2}\right) \theta\left(\frac{s'+s-1+r}{2}\right) \theta\left(\frac{s'-s+r}{2}\right) \theta\left(\frac{s'-s+1-r}{2}\right) \right/ \theta(1-r).
$$
Define the completed double Eisenstein series as in \eqref{dbleis20}
and write
$$
U(z;s,s'):=\sum_{j=1 \atop u_j \text{ even}}^\infty
    L^*(u_j,s+s'-1/2)L^*(u_j,s'-s+1/2) \frac{u_j(z)}{\s{u_j}{u_j}}.
$$
As in the last section, $\Omega$ and $U$ have exponential decay as $|r|$ and $|\Im(s_j)| \to \ci$.
Specializing \eqref{ml} to $w=s+s'$, we have proved the next result.
\begin{lemma}
For $s \in B_\rho$, $\Re(s)>1$ and $\Re(s')>\rho+5$
\begin{multline}\label{ei}
\de^*(z;s,s')= 2\theta(s)\theta(s')E(z;s+s') + 2\theta(1-s)\theta(s')E(z,s'-s+1)\\
+ U(z;s,s') +\frac{1}{2\pi i} \int_{(1/2)} \Omega(s,s';r) E(z,r) \, dr.
\end{multline}
\end{lemma}
From this we show the following.

\begin{theorem} \label{pra} The completed double Eisenstein series
$\de^*(z;s,s')$ has a meromorphic continuation to all $s,s' \in \C$ and we have the functional equations
 \begin{eqnarray}
   \de^*(z;s,s') &=& \de^*(z;s',s), \label{fei1}\\
   \de^*(z;s,s') &=& \de^*(z;1-s,1-s'). \label{fei2}
 \end{eqnarray}
\end{theorem}
\begin{proof}
First note that \eqref{ei} gives the meromorphic continuation of $\de^*(z;s,s')$ to all  $s,s'$ with $s \in B_\rho$ and $\Re(s')>\rho+5$.
As in the proof of Theorem \ref{meromcont}, we see that the further continuation  in $s'$ is given by \eqref{ei} along with residues that are picked up as the line of integration is crossed: for $s \in B_\rho$ fixed and $\Re(s') \to -\infty$ the continuation of $\de^*(z;s,s')$ is given by \eqref{ei} plus each of the following
\begin{eqnarray*}
     2\theta(s)\theta(1-s')E(z,s-s'+1) & &  \text{when \ \ \ \ $\Re(s') < \sigma+1/2$,} \\
   -2\theta(1-s)\theta(s')E(z,s'-s+1) & &  \text{when \ \ \ \ $\Re(s') < \sigma-1/2$,} \\
  2\theta(1-s)\theta(1-s')E(z,2-s-s') & &  \text{when \ \ \ \ $\Re(s') < -\sigma+1/2$,}  \\
  -2\theta(s)\theta(s')E(z,s+s') & & \text{when \ \ \ \ $\Re(s') < -\sigma-1/2$.}
\end{eqnarray*}
We have therefore shown the meromorphic continuation of $\de^*(z;s,s')$ to all $s \in B_\rho$ and $s' \in \C$.
Hence,
for all $s'$ with $\Re(s') <-\rho-4$, say, we have
\begin{multline}\label{ei2}
\de^*(z;s,s')= 2\theta(1-s)\theta(1-s')E(z,2-s-s') + 2\theta(s)\theta(1-s')E(z,s-s'+1)\\
+ U(z;s,s') +\frac{1}{2\pi i} \int_{(1/2)} \Omega(s,s';r) E(z,r) \, dr.
\end{multline}
The functional equation \eqref{fei2} is a consequence of  the easily checked
symmetries $U(z;1-s,1-s') = U(z;s,s')$, $\Omega(1-s,1-s';r) = \Omega(s,s';r)$
and a comparison of \eqref{ei} and \eqref{ei2}. The equation \eqref{fei1} has a similar proof, or more simply follows from the definition \eqref{dbleis20}.
\end{proof}

\begin{prop}\label{prb}
For any even  Maass Hecke eigenform $u_j$ (as in \S \ref{pww}) and all $s,s' \in \C$
$$
\s{\de^*(\cdot;s,s')}{u_j} = L^*(u_j,s+s'-1/2)L^*(u_j,s'-s+1/2).
$$
\end{prop}
\begin{proof}
As in Proposition \ref{poi}, only  $U(z;s,s')$ in \eqref{ei} will contribute to the inner product.
\end{proof}

With Theorem \ref{pra} and Proposition \ref{prb}, we have completed the proof of Theorem \ref{cpl}.

\section{Double Eisenstein series for general groups} \label{slast}
We proved in \S \ref{conti} that for $\G=\SL(2,\Z)$ the holomorphic double Eisenstein series $\ei_{s,k-s}(z,w)$ may be continued to all $s,w$ in $\C$ and satisfies a family of functional equations. That proof  does not extend to groups where  Hecke operators are not available. To show the continuation of $\ei_{s,k-s,\ca}(z,w)$ for $\G$ an arbitrary Fuchsian group of the first kind we first demonstrate a generalization of Proposition \ref{k1k2}. Recall the definitions of $u$, $v$ in \eqref{uv} and $\varepsilon_\G$ in \eqref{eg}.

\begin{theorem} For $s$, $w$ in the initial domain of convergence and even $k_1$, $k_2 \gqs 0$ with $k=k_1+k_2$ we have
\begin{equation}\label{pgen}
\ei^*_{s,k-s,\ca}(z,w) = 2^{\varepsilon_\G-1}\pi_{hol}\left[ (-1)^{k_2/2}y^{-k/2}E^*_{k_1,\ca}(\cdot, 1-u)E^*_{k_2,\ca}(\cdot, 1-v)/(2\pi^{k/2})\right].
\end{equation}
\end{theorem}
\begin{proof}
Let $g \in S_k(\G)$ and set $\G'=\sa^{-1}\G \sa$. Then
\begin{multline} \label{mul}
  \s{\ei_{s,k-s,\ca}(\cdot,w)}{g} = \int_{\G' \backslash \H} \Im(\sa z)^k \overline{g}(\sa z) \ei_{s,k-s,\ca}(\sa z,w) \, d\mu z \\
   = \int_{\G' \backslash \H} y^k \frac{\overline{g}(\sa z)}{\overline{j}(\sa,z)^k}
   \sum_{ \delta \in B \backslash \G'}
   j(\delta,z)^{-k}
   \left[\sum_{ \g  \in B \backslash \G' \atop  c_{\g  \delta^{-1}}  >0}
\left( c_{\g  \delta^{-1}} \right)^{w-1} \left( \frac{j(\g,z)}{j(\delta,z)} \right)^{-s}\right]
    \, d\mu z.
\end{multline}
Since $g(\sa z) j(\sa, z)^{-k} \in S_k(\G')$ we have
$$
y^k \frac{\overline{g}(\sa z)}{\overline{j}(\sa,z)^k j(\delta,z)^k} = \Im(\delta z)^k \frac{\overline{g}(\sa \delta z)}{\overline{j}(\sa, \delta z)^k}.
$$
Note also that $j(\g,z)/j(\delta,z) =j(\g \delta^{-1}, \delta z)$. Hence \eqref{mul} equals
\begin{equation}\label{mul2}
  2^{\varepsilon_\G}\int_{\G_\infty \backslash \H} y^k \frac{\overline{g}(\sa z)}{\overline{j}(\sa,z)^k}
   \left[\sum_{ \g  \in B \backslash \G' \atop  c_{\g}  >0}
 \left(c_{\g}\right)^{w-1} j(\g,z)^{-s}\right]
    \, d\mu z.
\end{equation}
Writing
$$
\sum_{ \g  \in B \backslash \G' \atop  c_{\g}  >0}
 \left(c_{\g}\right)^{w-1} j(\g,z)^{-s} = \sum_{ \g  \in B \backslash \G'/ B \atop  c_{\g}  >0}
 \left(c_{\g}\right)^{w-1} \sum_{m \in \Z} j(\g,z+m)^{-s}
$$
and using the Fourier expansion of $g$ at $\ca$:
$
j(\sa,z)^{-k} g(\sa z) = \sum_{n=1}^\infty a_{g,\ca}(n) e^{2\pi i n z},
$
we get that
\begin{eqnarray*}
  \eqref{mul2} &=& 2^{\varepsilon_\G}\sum_{n=1}^\infty \overline{a_{g,\ca}}(n) \sum_{ \g  \in B \backslash \G'/ B \atop  c_{\g}  >0}
 \frac 1{\left(c_{\g}\right)^{s+1-w}} \int_0^\infty \int_{-\infty}^\infty y^{k-2} \frac{e^{-2\pi i n x -2\pi ny}}{(x+d_\g/c_\g + iy)^s} \, dx dy  \\
   &=&  2^{\varepsilon_\G} I_k(s) \sum_{n=1}^\infty \frac{\overline{a_{g,\ca}}(n)}{n^{k-s}} \sum_{ \g  \in B \backslash \G'/ B \atop  c_{\g}  >0}
 \frac {e^{2\pi i n d_\g/c_\g}}{\left(c_{\g}\right)^{s+1-w}}
\end{eqnarray*}
for
$$
I_k(s):= \int_0^\infty \int_{-\infty}^\infty y^{k-2} \frac{e^{-2\pi i  x -2\pi y}}{(x + iy)^s} \, dx dy.
$$
The inner integral over $x$ may be evaluated with a formula of Laplace \cite[p. 246]{ww}:
$$
\int_{-\infty}^\infty  \frac{e^{-2\pi i  x}}{(x + iy)^s} \, dx = e^{-2\pi y}\frac{(2\pi)^s}{\G(s) e^{s i \pi/2}}
$$
so that
$$
I_k(s)=  \frac{\G(k-1)}{(4\pi)^{k-1}}\frac{(2\pi)^s}{\G(s) e^{s i \pi/2}}.
$$
With \eqref{eoa} and, for example \cite[Chap. 3]{Iwsp}, we recognize
$$
\sum_{ \g  \in B \backslash \G'/ B \atop  c_{\g}  >0}
 \frac {e^{2\pi i n d_\g/c_\g}}{\left(c_{\g}\right)^{2s}} = \sum_{ \g  \in \G_\infty \backslash \G'/ \G_\infty \atop  c_{\g}  >0}
 \frac {e^{2\pi i n d_\g/c_\g}}{\left(c_{\g}\right)^{2s}} = \frac{Y_{\ca \ca}(n,s)}{\zeta(2s) n^{s-1}}.
$$
It follows that we have shown
$$
\s{\ei^*_{s,k-s,\ca}(\cdot,w)}{g}= 2^{\varepsilon_\G-1} \frac{\zeta(2-2u)\G(k-s)\G(k-w)}{(2\pi)^{2k-s-w}}  \sum_{n=1}^\infty \frac{Y_{\ca \ca}(n,1-v) \overline{a_{g,\ca}}(n) }{n^{k-s-v}}.
$$
Reasoning as in the proof of \cite[(2.10)]{DO2010} we also find, for all even $k_1$, $k_2 \gqs 0$ with $k_1+k_2=k$,
\begin{multline*}
\s{(-1)^{k_2/2}y^{-k/2}E^*_{k_1,\ca}(\cdot, 1-u)E^*_{k_2,\cb}(\cdot, 1-v)/(2\pi^{k/2})}{g} \\
= \frac{\zeta(2-2u)\G(k-s)\G(k-w)}{(2\pi)^{2k-s-w}} \sum_{n=1}^\infty \frac{Y_{\cb \ca}(n,1-v) \overline{a_{g,\ca}}(n) }{n^{k-s-v}}.
\end{multline*}
Since $\ei^*_{s,k-s,\ca}(z,w) \in S_k(\G)$  and $g \in S_k(\G)$ is  arbitrary,  \eqref{pgen} follows.
\end{proof}

\begin{cor}
The double Eisenstein series $\ei^*_{s,k-s,\ca}(z,w)$ has a meromorphic continuation to all $s$, $w \in \C$ and as a function of $z$ is always in $S_k(\G)$. It satisfies the functional equation
\begin{equation}\label{fe9}
    \ei^*_{k-s,s,\ca}(z,w) = (-1)^{k/2}\ei^*_{s,k-s,\ca}(z,w).
\end{equation}
\end{cor}
\begin{proof}
Since $E^*_{k,\ca}(z, s)$ has a well-known continuation to all $s \in \C$, due to Selberg, the continuation of $\ei^*_{s,k-s,\ca}(z,w)$ follows from \eqref{pgen}. The change of variables $(s,w) \to (k-s,w)$ corresponds to $(u,v) \to (v,u)$ and so \eqref{fe9} is also a consequence of \eqref{pgen}.
\end{proof}

If $\G$ has more than one cusp then $\ei^*_{s,k-s,\ca}(z,w)$ does not appear to possess a functional equation of the type \eqref{ll2} as $(s,w)\to (w,s)$. This corresponds on the right of \eqref{pgen} to $(u,v) \to (u,1-v)$ and the functional equation for $E^*_{k_2,\ca}(\cdot, 1-v)$ involves a sum over cusps as in \eqref{sucu}.

We remark that the functional equation \eqref{fe9} also follows directly from \eqref{dbleis4} if $-I \in \G$: replace $\g$ and $\delta$ in the sum by $-\delta$ and $\g$ respectively.

Finally, it would be interesting to find the  continuation in $s$, $s'$ of the non-holomorphic double Eisenstein series $\de^*_\ca(z;s,s')$ for general groups. We expect that a similar decomposition to \eqref{ei} should be true.

{\footnotesize \bibliography{kernelb2}
\vskip 3mm
\noindent
{\sc School of Mathematical Sciences, Univ. of Nottingham,  University Park, Nottingham NG7 2RD, U.K.}
\newline
{\it E-mail address: }{\tt nikolaos.diamantis@maths.nottingham.ac.uk}
\newline
\vskip 0mm
\noindent
{\sc Dept. of Mathematics, The CUNY Graduate Center,  365 Fifth Ave., New York, NY 10016-4309, U.S.A.}
\newline
{\it E-mail address: }{\tt cosullivan@gc.cuny.edu}
}

\end{document}